\documentclass[11pt]{amsart}
\usepackage{amsmath,amssymb,amsthm,amsfonts,color,longtable}
\usepackage{hyperref}
\usepackage{amssymb,longtable}
\usepackage{amsfonts}
\DeclareRobustCommand{\SkipTocEntry}[5]{}
\usepackage{setspace}
\usepackage{filecontents}
\usepackage [english]{babel}
\usepackage [autostyle, english = american]{csquotes}
\MakeOuterQuote{"}
\usepackage{amsmath} 
\usepackage{hyperref} 
\usepackage{latexsym}
\usepackage{array}
\usepackage{amssymb}
\theoremstyle{plain}

\newtheorem{assumption}{Assumption}
\numberwithin{equation}{section}
\newtheorem{thm}{Theorem}[section]
\newtheorem{prop}[thm]{Proposition}
\newtheorem{lemm}[thm]{Lemma}
\newtheorem{cor}[thm]{Corollary}
\newtheorem{defi}[thm]{Definition}

\theoremstyle{remark}
\newtheorem{rema}[thm]{Remark}
\newtheorem{defn}[thm]{Definition}
\renewcommand{\div}{{\rm div}}
\newcommand{\curl}{{\rm curl}} 
 
\newcommand{\Z}{\mathbb{Z}}  
  
\newcommand{\N}{\mathbb{N}}

\newcommand{\R}{\mathbb{R}}

\newcommand{\Supp}{\rm Supp }
\newcommand{\wt}{\widetilde}
\newcommand{\D}{\Delta}
\newcommand{\bq}{\begin{equation}}
\newcommand{\eq}{\end{equation}}
\begin{document}

\title[]{ $L^\infty$ ill-posedness for a class of equations arising in hydrodynamics}

\author{Tarek M. Elgindi and Nader Masmoudi}
\address{The Courant Institute,\\ New-York University\\
251 Mercer Street\\   USA}
\email{elgindi@cims.nyu.edu and masmoudi@cims.nyu.edu}

\thanks{
Both authors were  partially supported by NSF grant DMS-1211806.}
%T.M. Elgindi is partially supported by NSF grant \# XXXX grant.}

\maketitle

\begin{abstract}
Many questions related to  well-posedness/ill-posedness in critical spaces 
for hydrodynamic equations have been open for many years. In this article we give a new approach to studying norm 
inflation (in some critical spaces)  for a wide class of equations arising in hydrodynamics. As an application, we prove strong ill-posedness of the $n$-dimensional Euler equations in the class $C^1\cap L^2 (\Omega)$ and also in $C^k \cap L^2(\Omega)$ where $\Omega$ can be the whole space, a smooth bounded domain, or the torus. We also apply our method to the Oldroyd B, surface quasi-geostrophic, and Boussinesq systems.\end{abstract}

\tableofcontents

\section{Introduction}

\addtocontents{toc}{\SkipTocEntry}
\subsection{The concept of well-posedness}
In 1903, Jacques Hadamard set forth a concept of well-posedness for partial differential equations of physical origin. Hadamard suggested that for a PDE problem to be well-posed (whether it be an initial value problem, a boundary value problem, or both) it should enjoy the following properties:

(1) Existence,

(2) Uniqueness,

(3) Continuous dependence on the initial/boundary data.

These conditions were obviously physically motivated because if an equation is to model a physical phenomenon then one would expect that solutions to the model exist, are unique, and that small perturbations should not result in chaotic responses by the system (at least for small times). 
Of course, well-posedness or the lack thereof depends upon which space one is looking for a solution in; thus a particular initial/boundary value problem might be well-posed in  one solution class but ill-posed in another. 
 
Based on the definition of well-posedness, one can think of at least three types of ill-posedness: nonexistence, nonuniqueness, and discontinuous dependence on the data.
In this work we are interested in nonexistence  and discontinuity with respect to the data.
In our investigations nonexistence in a space $X$ 
will be deduced from the fact  that a solution uniquely exists in a 
larger space $X_0$ with $X$ norm that becomes immediately equal to infinity.  
Nonexistence can be 
thought of as  the strongest kind of ill-posedness. 

There are weaker kinds of ill-posedness  that were studied in the literature: 
The solution map may not be $C^1$ or $C^2$ with respect to the data (see, for example, \cite{FRV} and \cite{MST02}), 
or the the solution map is not uniformly 
continuous with respect to the data in a bounded set \cite{KT05}. 
Many other  ill-posednessed questions have been studied in the case of 
dispersive equations and we refer the reader to \cite{CCT03,Lebeau05,IMM11,AC09,Carles07}. 
%  and these are the two kinds of ill-posedness we are concerned with in this work.  
\begin{defn} 
We say that a Cauchy problem
$$f_{t} = N(f),$$
$$f(0)=f_0$$
is \emph{mildly ill-posed} in a space $X$ if there exists a space $Y$ continuously embedded in $X$  and a constant $c>0$ such that 
for all $\epsilon>0,$ there exists  $f_{0} \in Y,$ with $$|f_{0}|_{X} \leq \epsilon$$ for which  
there exists a unique solution, $f(t) \in L^\infty ([0,T]; Y)$ for some $T>0$ with initial data $f_{0}$ such that: $$|f(t)|_{X} \geq c$$ for some $0<t<\epsilon.$ If $c$ can be taken to be equal to $\frac{1}{\epsilon}$,  we will say that the equation is \emph{strongly ill-posed.}  

\end{defn}

Typically, the space $Y$ will be a space which is smoother than $X$ and for which local existence 
is already known and $Y\subset X$.  The initial data $f_0$ will be chosen such that $|f_{0}|_{X} \leq \epsilon$  while 
  $|f_{0}|_{Y}$ may be large.  
Of course, both mild ill-posedness and strong ill-posedness only imply discontinuity with respect to the initial data. However, in many cases, strong ill-posedness can be used to prove nonexistence (the strongest form of ill-posedness in the sense of Hadamard)--see section 9.   

\begin{rema}
We remark here that there are several different understandings of the term ill-posedness--particularly that it cannot be used to refer to the sort of borderline phenomenon that is in this paper. While classically it might be more proper to use ill-posedness to refer to genuinely ill-posed problems such as the backwards heat equation, or the Kelvin-Helmholtz problem in vortex sheets, and other such situations where the issue isn't a matter of critical spaces but a real lack of solvability or uniqueness. Our reasons for using the term ill-posedness in this context are twofold: first that it has become common in the literature to use the term to refer even to non-solvability in critical spaces. Second, we actually do include results on genuine ill-posedness such as the result given in Proposition \ref{3dBahouriChemin} below where we construct a solution of the 3D Euler equations which belongs initially to $W^{1,p}$ for all $p<\infty$ but which leaves $W^{1,q}$ for $q>2$ in finite time. 
\end{rema}

\addtocontents{toc}{\SkipTocEntry}
\subsection{The Euler equations of incompressible flow} 
One of the most elusive and difficult issues to deal with in studying the equations of incompressible hydrodynamics is their inherent nonlocality. This can be seen intuitively: every part of the fluid should, in some way, affect every other part of the fluid. Recall the Euler equations for inviscid and incompressible flow modeling an ideal (frictionless) liquid in the whole space:

\begin{equation} 
\partial_{t} {u} + ({ u}\cdot \nabla) {u} +\nabla p =0 \, \, \, \text{on} \, \, \mathbb{R}^{n} \times (0,\infty),\end{equation}
\vspace*{-7mm}
\begin{equation} \text{div}({ u})=0 \, \, \, \text{on} \, \, \mathbb{R}^{n} \times (0,\infty), \end{equation}
\vspace*{-7mm}
\begin{equation} { u}(x,t) \rightarrow 0 \, \, \, \text{as}\, \, \, |x|\rightarrow \infty, \end{equation}
\vspace*{-7mm}
\begin{equation} { u}(x,0) =  u_0\, \, \text{on} \, \, \mathbb{R}^{n}. \end{equation}

In (1.1)-(1.4), ${ u}(x,t) \in \mathbb{R}^{n}$ is the velocity of the fluid at position $x \in \mathbb{R}^{n}$ and at time $t\in [0,\infty).$ Equation (1.4) says that the initial velocity profile of the fluid is given by $ u_0 .$ Equation (1.3) is an idealized condition which says that the fluid is at rest at spatial infinity. (1.2) dictates that the fluid be incompressible, which means that if one tracks the evolution of a particular portion of the fluid in time then the volume of that portion cannot change in time. Equation (1.1) is just Newton's second law, the momentum equation, which says that the only force acting on the fluid is that of internal pressure. 
 
One of the most challenging basic problems in the study of fluid equations, is the question of well-posedness for the Euler equations. In two dimensions, the global well-posedness question was settled in $C^{k,\alpha},$ spaces with $ k \geq 1, 0<\alpha<1$ by Wolibner \cite{Wolibner33} and H\"older \cite{Hol} in the 1930's. Note that well-posedness in $C^k$
was left open. In three space dimensions, it is not known whether the Euler equations are globally well-posed in the class of smooth solutions. The main results which exist in this direction are \emph{local} well-posedness results which go back to Lichtenstein \cite{Lich} in the $C^{k,\alpha}$ case. There is also a literature on blow-up criteria such as the blow-up criteria of Beale, Kato, and Majda \cite{BKM} and the geometric criteria of Constantin, Fefferman, and Majda \cite{ConstFM} (see \cite{DengHouYu} and the references therein for various improvements on these blow-up criteria). The criteria of Beale-Kato-Majda states that the growth of the high Sobolev norms of the velocity field is controlled by the growth of the $L^\infty$ norm of the vorticity. In particular, they prove:
$$|u|_{H^s} \leq |u_0|_{H^s}^{\exp(C\int_{0}^t |\curl(u)|(\tau)_{L^\infty}d\tau)}, \text{for all} \, \, s>\frac{d}{2}+1,$$ 
which was an improvement on the classical energy estimate:
$$ |u|_{H^s} \leq |u_0|_{H^s}\exp(C\int_{0}^t |\nabla u(\tau)|_{L^\infty}d\tau).$$
Note that the Beale-Kato-Majda estimate allows one to say that $L^\infty$ bounds on the curl of $u$ prevent blow-up as opposed to the classical energy estimate which requires $L^\infty$ bounds on the full gradient of $u$. This improvement, however, comes at the cost of an exponential so that the Beale-Kato-Majda estimate is of double exponential type. 

In $n$ dimensions, the equation for $\omega:=\nabla\times u$ is:
$$\partial_t \omega + (u\cdot \nabla)\omega=\omega\cdot\nabla u.$$
Note that if $\omega$ depends only on two variables and it has no third component, then $\omega\cdot\nabla u\equiv 0$ which immediately implies the conservation of $\omega$ in all $L^p$ spaces including $L^\infty$. This is the main tool in proving global existence and uniqueness for smooth data in two dimensions. Upon this basis, Yudovich built a well-posedness theory for the 2d Euler equations with vorticites in $L^\infty$. 
While the boundedness of the vorticity gives us global well-posedness in 2d, the equation is barely well-posed in the sense that Holder norms of $\omega$ are allowed to grow double exponentially in time as is given in the Beale-Kato-Majda estimate above.  
The situation is even worse in higher dimensions, $n\geq 3$, since the term $\omega\cdot\nabla u$ can actually cause vorticity growth. 
In this regard, it is instructive to mention the following interesting example. 

\addtocontents{toc}{\SkipTocEntry}
\subsection{Two examples illustrating the difference between the 2d and 3d Euler equations}

We begin by introducing the so called 2$\frac{1}{2}$-dimensional solutions of the 3d Euler equations. Take initial data for the 3d Euler equations of the following form: $$u_0(x,y,z)=(u^1_0(x,y),u^2_0(x,y),u^3_0(x,y)).$$
Then there exists a solution $u(t)$ which remains a function of $x$ and $y$ only and which solves:
$$\partial_t u^h +u^h\cdot\nabla u^h +\nabla p=0,$$
$$\partial_t u^3 + u^h\cdot\nabla u^3=0,$$
with $u^h=(u^1,u^2)$. Hence, $u^h$ satisfies the 2d Euler equations with initial data $u^h_0$ and $u^3$ is just passively transported by $u^h$. In particular, the vorticity associated to $u^h$ satisfies the 2d vorticity equation:
$$\partial_t \omega^h +u^h\cdot\nabla \omega^h =0.$$
This means, using the method of Yudovich, that if $\omega_0^h\in L^\infty(\mathbb{T}^2)$ is mean-zero there exists a unique solution to the 2d vorticity equation $\omega^h(t)$ which remains in $L^\infty_{x,t}$. From there, we know that the corresponding $u^h$ is log-Lipschitz due to the fact that it is divergence-free and its curl, $\omega_h$, is bounded . This allows us to solve the transport equation for $u^3$ uniquely. However, $u^3$ could potentially lose regularity since $u^h$ is not Lipschitz but only log-Lipschitz. 

\begin{prop} 

\label{3dBahouriChemin}

There exists $u_{0}=(u^h_0,u^3_0) \in W^{1,2}(\mathbb{T}^3)$with $\omega_{0}=\nabla \times u_0 \in L^\infty(\mathbb{T}^3)$ depending only on two variables, $x$ and $y$, and a unique solution $u(t)$ of the 3d Euler equations which satisfies the following properties:

\begin{enumerate}

\item $u$ is a function of $x,y,$ and $t$ only.
\item $|\omega^h(t)|_{L^\infty}=|\omega_0^h|_{L^\infty}$, for all $t>0$. 
\item $u(t)\in C^{\exp(-Ct)}$ for some large constant $C>0$.

Moreover, $$| \omega(t)|_{L^p}=+\infty$$

for all $p>\frac{2}{1-e^{-t}}.$

\end{enumerate}

\end{prop}

As a consequence, the 3D Euler equations are strongly ill-posed  (in the strongest sense of Hadamard) in the class of finite-energy velocity fields with bounded vorticity. The example is quite simple: if $u^{h}$ is taken to be such that  $\omega^h=sgn(x)sgn(y),$ then it is known (\cite{BC94}) that the Lagrangian flow map associated to $u^{h}$ is of regularity $C^{e^{-t}}$ and no better. This is due to the fact that $\nabla u^h$ has a logarithmic singularity at the origin. Thus, $u^{3}(x,t) =u^{3}_{0}(\Phi_h (x,-t))$ and $u^3_0$ is then chosen such that $u^3$ does not belong to $C^{\alpha}$ for $\alpha>e^{-t}$ (for example, one can take $u^{3}_0(x,y) =x^2+y^2$ in a small neighborhood of the origin). Proposition \ref{3dBahouriChemin} then follows by Sobolev embedding.  We remark that this example can be modified to give an example of a global smooth solution of the 3D Euler equations on $\mathbb{T}^3$ for which the vorticity grows exponentially in time:

\begin{prop}

There exists $u_0\in C^\infty (\mathbb{T}^3)$ such that the 3-D Euler system has a global strong solution $u(t)$ with initial data $u_0$ for which:

$$|\omega(t)|_{L^\infty}\geq e^{t}$$ for all $t>0.$
\end{prop}

To prove the proposition, take $u_h$ to be the stationary solution of the 2-D Euler equations $(sin(x)cos(y),-cos(x)sin(y)).$ Note that the flow-map associated to $u_h$ is hyperbolic at the origin. Indeed, linearizing, we see that $u_h \approx (x,-y)$ near the origin. As a consequence, the flow map induced by $u_h$ has an exponential contraction along the y-axis at the origin. Once $u_{3}$ is chosen to be non-constant along the y-axis the exponential growth is attained. Using some recent ideas of Kiselev and Sver\'ak \cite{KS} one can actually prove that the exponential growth for data with a hyperbolic point for the $2\frac{1}{2}$ dimensional solutions is, in a sense, generic. This gives some idea as to why one might consider questions of ill-posedness for weak solutions: behind ill-posedness for weak solutions there may be uncontrollable growth for strong solutions. 

\addtocontents{toc}{\SkipTocEntry}
\subsection{Previous ill-posedness results for weak solutions}

In recent years, the question of well-posedness at low-regularity has become of great interest due to its connections with turbulence and the essence of weak solutions. As we stated above, ill-posedness can mean one of three things: the initial data can start in $X$ and then leave $X,$ or we have non-uniqueness, or the solution map is discontinuous--the first case being the strongest form of ill-posedness.  There are still many questions which are unanswered in the well-posedness theory of weak solutions even in two spatial dimensions. Existence of weak solutions in the class $W^{1,p},$ $p>1$ was established in two dimensions. Uniqueness has been proven only for weak solutions which are Lipschitz or "almost" Lipschitz (see the works of Yudovich \cite{Yudovich63}, \cite{Yudovich95}, and Vishik \cite{V1}, \cite{V2} for example). 

Previous works in the direction of ill-posedness include results of DiPerna-Lions \cite{DL89im} (nonexistence for $u\in W^{1,p}, p<\infty$ in three dimensions), Bardos-Titi \cite{BT} (nonexistence for $u\in C^\alpha$, $\alpha<1$ in three dimensions), Misiolek-Yoneda \cite{MisY2} (nonexistence in critical Besov-spaces in three dimensions), and Cheskidov Shvydkoy \cite{CheskidovShvydkoy} (non-continuity of the solution map in supercritical Besov spaces) to mention a few. All of the above cases except the last one were done by explicit examples.  A very interesting preprint by Bourgain and Li \cite{BLi} studies the ill-posedness of the Euler equations with velocity in $H^{\frac{d}{2}+1},$ the difficulty being that $H^{\frac{d}{2}+1}$ is a critical space sitting at the lower threshold of the classes of strong solutions where local well-posedness holds (the Euler equations are locally well-posed in $H^s$ with $s>\frac{d}{2}+1).$ 
Finally, nonuniqueness of weak solutions was shown by Scheffer \cite{Scheff}, Shnirelman \cite{Shnir}, De Lellis and Szekelyhidi \cite{DeLS}, Isett \cite{Isett}, and Buckmaster \cite{DeLSBu} for weak solutions of the Euler equations in various "very weak" spaces, the smallest of which is $C^{\alpha}_{t,x},$ $\alpha<\frac{1}{5}.$ It is conjectured that non-uniquness should hold up to $C^\frac{1}{3}.$

\addtocontents{toc}{\SkipTocEntry}
\subsection{How ill-posedness in $C^1$ arises}

We will now move to discuss some of the ideas behind how ill-posedness for $C^1$ initial velocity can be understood for the 2d Euler equations. We would like to clarify from the outset: the phenomenon we are about to describe is not special to dimension 2,3, or any dimension. Indeed, if we are able to prove ill-posedness in $C^1$ in two dimensions it will automatically imply ill-posedness in higher dimensions in many settings.\footnote{For example, we can say that any solution of the 2d Euler equations on $\mathbb{T}^2$ is a solution of the 3d Euler equations on $\mathbb{T}^3$ using ideas similar to those in Subsection 1.3. Furthermore, by passing to axi-symmetric solutions one can usually say the same for solutions on $\mathbb{R}^3$.} 

Now, to prove ill-posedness in $C^1$ we need to consider the equation for $\nabla u$. 
By differentiating (1.1) we get:
$$\partial_t \nabla u + (u\cdot \nabla)\nabla u + Q(\nabla u, \nabla u)+ D^2 p=0,$$
with $Q$ just a bilinear quadratic form. As far as local-in-time $L^\infty$ estimates go, $Q$ and the transport term cause no difficulty. The question is whether $D^2p$ can be controlled by a function of $|\nabla u|_{L^\infty}$ and weaker norms.  By taking the divergence of (1.1) and using that $\div(u)=0$ we get:
$$\Delta p = -2\, \text{det}(\nabla u).$$
This means that $D^2 p$ is some singular integral matrix applied to $\text{det} (\nabla u)$. It turns out to be possible to construct a specific function $u\in C^1$ for which $D^2 p\not\in L^\infty$.  This is the source of the ill-posedness. However, even if at the initial time $D^2 p$ is unbounded, the difficulty then moves to proving that if we solve the Euler equations with some special initial data, the $C^1$ norm of $u$ becomes unbounded--though bounded initially. 

The difficulty in proving that the full non-linear problem actually exhibits this growth is mainly that the transport term could prevent this growth. However, as is explained in the next subsection, this is not possible when the velocity field belongs to the class $L^\infty([0,T]; \text{Lip})$. Hence, the general method for dealing with such problems is:

\begin{enumerate}
\item{Construct initial data which is designed to cause growth at $t=0$.}
\item{Prove that the non-linear evolution does not eliminate the growth mechanism for $t>0$.}
\end{enumerate}
The most difficult part in the $C^1$ problem for the Euler equations is part (2). 
This motivates the study of the following simple linear problem.

\addtocontents{toc}{\SkipTocEntry}
\subsection{The general result of the paper and some applications}
The main result of this paper is that general transport equations of the form:
\begin{equation} \partial_t f+u\cdot\nabla f = R(f),\end{equation}
with $u$ a Lipschitz continuous, divergence-free, velocity field and $R$ a singular integral operator, are strongly ill-posed on $L^\infty$ under some mild conditions on $R$ and in any spatial dimension. 
The idea is to consider (1.5) as a perturbation of the equation $$\partial_t f =R(f). $$ 
It must first be proven, under some conditions on $R$, that this equation itself is ill-posed on $L^\infty$. Since $R$ is linear this is possible to do for short time simply by writing the solution as a series expansion in powers of $tR$. 
However, when dealing with (1.5), since $u\cdot\nabla f$ contains derivatives of $f$ and since we are working at such low regularity, it is not possible to take the transport term as a perturbation directly. Hence, we opt to "put the derivative on $u$" by passing to the Lagrangian formulation, which is to write the equation along the characteristics of $u$. Indeed, we solve $$\frac{d}{dt} \Phi =u\circ \Phi,$$ $$\Phi\Big|_{t=0}=Id,$$
Then we write: 
$$\partial_t (f\circ \Phi)=R(f)\circ \Phi. $$
If $\Phi$ commuted with $f$ we would be able to conclude the ill-posedness. Hence we write:
$$\partial_t (f\circ\Phi)= R(f\circ\Phi)+ [R,\Phi]f,$$ 
with $$[R,\Phi] f=R(f)\circ\Phi-R(f\circ\Phi).$$
Thus the result relies upon proving good estimates for the commutator $[R,\Phi]$. It turns out that this commutator can be cast into the framework of the so-called Calder\'on commutators and we are able to prove that the operator norm of $[R,\Phi]$ on a suitably chosen space $X$ can be bounded as follows:
$$|[R,\Phi]|_{X\rightarrow X} \lesssim \max\{|\Phi-Id|_{Lip}, |\Phi^{-1}-Id|_{Lip}\}. $$ 
This allows us to prove $L^\infty$ ill-posedness for equations like (1.5) in great generality. 

\addtocontents{toc}{\SkipTocEntry}
\subsection{Application to the Euler equation in $C^k$ spaces}
As an application we then prove:

\begin{thm} Let $\Omega$ be either $\mathbb{R}^2,$ $\mathbb{T}^2$, or a smooth bounded domain. 

\item{(I) There exists a divergence-free $u_0\in C^1_{c}(\Omega)$ so that the unique (Yudovich) solution of the incompressible 2d Euler equations does not belong to $L^\infty_t([0,\delta]; C^1(\Omega))$ for any $\delta>0$.}

\item{(II) For each $k=2,3,...$ there exists a divergence-free $u_0\in C^k(\Omega)$ so that the unique solution of the incompressible 2d Euler equations belonging to $C^{k-\epsilon}(\Omega)$ for every $\epsilon>0$ satisfies: $$|D^{k} u|_{L^\infty}=+\infty,$$ for all $t\in [0,1].$}
\end{thm}

\begin{rema}
\item{(1) The initial data in both parts of Theorem 1.5 can be taken to be as small as we want.}
\item{(2) It is not difficult to extend these results to the case of $u_0\in C^1(\mathbb{R}^3)$ with decaying data by using the axi-symmetric 3d Euler equations, but we do not do this here. The extension to higher dimensional periodic data is obvious.}
\item{(3) Note that in the case (II) we have a much stronger result because the transport term is much easier to deal with since the velocity field almost belongs to $C^2.$}
\end{rema}

The proof of Theorem 1.5 basically reduces to casting the incompressible Euler equations in the setting of our general linear equation (1.5) and then checking that the conditions on $R$ are satisfied.

\addtocontents{toc}{\SkipTocEntry}
\subsubsection{The right initial data } 
 
We would like to make a few comments on the initial data which leads to growth. The reason we choose to highlight this particular issue here is that understanding how the initial data is constructed may give some insight into other questions such as the growth of Sobolev norms in the 2d case or even the finite time blow up problem in 3d. Indeed, the recent result of Kiselev and Sver\'ak \cite{KS} is based upon perturbing a particular singular stationary solution constructed by Bahouri and Chemin. Moreover, it must be noted that Sobolev norms for smooth solutions in 2d grow double exponentially if and only if\footnote{One direction of the equivalence is clear using the energy inequality $$|u|_{H^s}\leq|u_0|_{H^s} \exp ( \int_0^t |u(s)|_{C^1}ds).$$ The other direction is a consequence of the vorticity conservation in $L^\infty$ and the inequality: $$|u|_{C^1}\leq C |\omega|_{L^\infty} \log (|u|_{H^s}+10)$$ for $s>2.$} the $C^1$ norm of $u$ grows exponentially. Hence, growth of the $C^1$ norm is crucial. 

It turns out that the right initial data to choose to prove ill-posedness in $C^1$ is precisely a $C^1$ function $u^*$ for which the solution $p^*$ of $$\Delta p^* = -2 \text{det} (\nabla u^*)$$ satisfies that $D^2 p^* \approx Log|x^2+y^2|$ near $(0,0).$ Define $$G= (xy^3-yx^3) Log(x^2+y^2).$$ Let $$u^*= \nabla^\perp \Delta G + \kappa (y,0),$$ for some $\kappa>0$ sufficiently large. 

Because $xy^3-yx^3$ is harmonic, $\nabla u^* \in L^\infty$. At the same time, $\partial_{xyyy}G\not\in L^\infty.$ By choosing $\kappa$ large enough, it can then be shown that $$|D^2p^*|\approx Log(x^2+y^2).$$

\addtocontents{toc}{\SkipTocEntry}
\subsubsection{Comparison with some recent results}

This paper is devoted to establishing a general framework to study non-linear  and non-local transport equations in critical spaces based on $L^\infty$ (such as $L^\infty$, $C^1$, etc.). As an application, we prove non-existence of $C^1$ solutions to the incompressible Euler equations coming from $C^1$ data (this is done in Section 8). Recently, the same result was proven by Bourgain and Li \cite{BLi2} using completely different methods (it should be noted that this article and \cite{BLi2} actually appeared on the arXiv on the exact same day). Misiolek and Yoneda also proved non-continuity of the solution map on $C^1$ \cite{MisY}. A main difference between the method of Bourgain and Li in \cite{BLi2} and our method is that they focus upon the vorticity formulation of the Euler equations and see the ill-posedness through the fact that if $R$ is a singular integral operator and $\Phi$ is a $C^1$ map, $R(f)\in L^\infty$ does not imply that $R(f\circ \Phi)$ belongs to $L^\infty$. Hence, the source of ill-posedness for them is very much Lagrangian in nature. Our main approach is to observe that, in the Eulerian frame, the second derivative of the pressure is a singular integral of the gradient of the velocity field which isn't bounded in general on $L^\infty$. That said, we do incorporate certain elements of the Lagrangian description as well. Another important point is that while the results in \cite{BLi2} are exclusively focused on the Euler equations, our outlook is slightly more general and the application to the Euler equations is just one part of this work. Outside of fluids, there is a substantial literature on ill-posedness type results in the study of dispersive equations (see, for example \cite{CCT03}, \cite{KT05}, and \cite{IMM11}). We emphasize, however, that there is a key difference between those results and the results presented here. The ill-posedness in those cases are generally non-linear and in super-critical settings. Here, the ill-posedness is \emph{linear} and at critical regularity. In fact, the main difficulty in our work is to prove that the non-linearity does not serve as a stabilizing mechanism as happens, for example, in the work of Bressan and Nguyen \cite{BressanNguyen} where $L^\infty$ ill-posedness at the linear level is quenched at the non-linear level (this is discussed in more detail below).

\addtocontents{toc}{\SkipTocEntry}

\subsection{Further Applications}
In Section 10 we give a few extra applications of the main commutator estimate, Proposition \ref{CommutatorEstimate}. In particular, we prove $L^\infty$ and $W^{1,\infty}$ estimates for three system: the Oldroyd-B system, the surface quasi-geostrophic (SQG) system, and the Boussinesq system. Each one of these systems has a special extra feature which gives a different twist on how we use Proposition \ref{CommutatorEstimate}. In the Oldroyd-B case, the "singular integral" we encounter is actually non-linear since it is based on the operator $f\rightarrow u$ where $u$ is the solution of the Navier-Stokes equation with forcing term $f$ and initial data $u_0\equiv 0$. Using some a-priori estimates specific to the Oldroyd-B system, one can effectively recast the co-rotational Oldroyd-B system into a setting where we can directly apply Proposition \ref{CommutatorEstimate}. This is done in Section 10.1. We also consider the SQG system where the linear singular integral term which leads to the $L^\infty$ ill-posedness is actually variable coefficient. In that case, we cannot write the solution operator on the Fourier side as we do in the constant coefficient case. Nevertheless, we circumvent this by doing a Taylor expansion in time using local well-posedness in a critical Besov space and thus succeed in proving the (mild) ill-posedness. 

\addtocontents{toc}{\SkipTocEntry}
\subsection{Brief outline of the paper}

In section 2 we will give our main technical linear result from which all the applications will follow. Section 3 will be the proof of the linear estimate which relies on a highly non-trivial commutator estimate on Besov spaces. In Section 4 we will prove ill-posedness for linear transport equations with singular integral forcing. In sections 5 and 6 we will prove ill-posedness for perturbations of the 2D Euler equations and the full 3D Euler equations in vorticity form in the class of flows with bounded vorticity. In section 7 we prove the strong ill-posedness of the Euler equations in the class of $C^1$ velocity fields and that the solution map on $C^{1,\alpha}$ does not have a bounded extension to $C^1.$ In section 8, we show non-existence of a $C^1$ solution to the Euler equations from some $C^1$ initial data.   In Section 9 we outline how to modify the the proof in the $C^1$ case to prove ill-posedness in $C^k$ for any integer $k\geq 1$. Finally, in Section 10, we apply our linear result to the Oldoryd-B, surface quasi-geostrophic, and Boussinesq systems. 

\section{The Building Block : A  Linear Result}

We consider linear partial differential equations of the following form:

\begin{equation} \label{LinearRTP1} f_{t}+u\cdot\nabla f=R (f), \end{equation}
\begin{equation} \text{div}(u)=0, \end{equation}
\begin{equation} f(t=0)=f_{0}. \end{equation}

Here, $R$ is a Calder\'{o}n-Zygmund singular-integral operator and $u$ is a divergence-free Lipschitz function belonging to the class $L^\infty([0,T]; \text{Lip})$ for some $T>0$. 
A natural question to ask is: 
\begin{equation} \label{question} \text{If $f_{0} \in L^\infty,$ is it true that $f(t)\in L^\infty$ for even a short time?} \end{equation}
This question can be answered affirmatively if $L^\infty$ were replaced by $L^p$ due to the divergence free condition on $u$ and the fact that Calder\'on-Zygmund singular-integral operators are bounded on $L^p$. Since such operators are generally unbounded on $L^\infty$, the answer to (2.4) might be negative.  

\addtocontents{toc}{\SkipTocEntry}

\subsection{A first example}

It is instructive to consider the following one-dimensional example when the velocity field is not present ($u\equiv 0$):

$$f_{t}=H(f),$$ where $H$ is the Hilbert transform. Recall that the Hilbert transform is a singular integral operator mapping $L^2(\mathbb{R})$ to $L^2(\mathbb{R})$ which is defined in the following way: $$H(f)(x):=\frac{1}{\pi}P.V.\int_{\mathbb{R}} \frac{f(y)}{x-y}dy.$$ It is standard that $H$ has the following form under the Fourier transform: $$\mathcal{F}( H(f))(\xi)= -i sgn(\xi) \mathcal{F}(f)(\xi).$$ Recall also the Riesz transforms, $R_j$, $j=1,...,n$, which are higher dimensional analogues of the Hilbert transform. These are given under the Fourier transform, by multiplication by $-i\frac{\xi_j}{|\xi|}.$

Simple calculations, using the fact that $H^2=-1$ yield that the solution to this evolution equation can be written explicitly:
$$f(t):=\exp(t H) f_0=\text{cos}(t)f_{0}+\text{sin}(t)H(f_{0}).$$
It is known that while $H:L^p\rightarrow L^p$ is bounded for each $1<p<\infty$, there exist functions $f\in L^\infty$ for which $H(f)\not\in L^\infty$. 

There are two points which are important to take from this calculation:
\\
(1) If $t$ is small enough, $\exp(tH)$  is unbounded on $L^\infty.$ 
\\
(2) If $t$ is small enough $\exp(tH)$ \text{ is as singular on} $L^\infty$ \text{as} $tH$  is. 
\\
Hence, the answer to question (2.4) is negative for the case $u\equiv 0$ and $R=H.$

\addtocontents{toc}{\SkipTocEntry}

\subsection{A second example}\label{SecondExample}

A second useful example is the so-called Burgers-Hilbert equation:
\begin{equation} \label{BH} \partial_t u + u\partial_x u =H(u),\end{equation} where $H$ is the Hilbert transform.  This equation, while non-linear can be seen as a one-dimensional version of (2.1) where now the advected quantity and the velocity field are the same function and $R=H$. While $L^\infty$ initial data can leave $L^\infty$ for the linearized equation $$\partial_t u=H(u),$$ Bressan and Nguyen \cite{BressanNguyen} showed the propagation of $L^1(\mathbb{R})\cap L^\infty(\mathbb{R})$ solutions for \eqref{BH}. In other words, given $u_0\in L^1\cap L^\infty$, there exists a unique entropy solution of \eqref{BH} with initial data $u_0$ and which remains bounded in $L^1\cap L^\infty$. Hence, for the Burgers-Hilbert equation, the answer to \eqref{question} is affirmative, even though the answer for the linearized equation is negative. What we learn from this example is that the transport term can prevent growth of the $L^\infty$ norm if it is strong enough. In the following sections, we will prove that if $u$ is uniformly Lipschitz continuous on a time interval $[0,T]$, then the linear equation (2.1) must exhibit $L^\infty$ growth under some mild  conditions on $R$. This will then be applied to some non-linear problems. 

\addtocontents{toc}{\SkipTocEntry}

\subsection{Preliminaries and notations}

Before stating our main result we will first introduce a little bit of necessary background material. Due to the criticality of our problem, we will \emph{need}  Besov spaces.  Throughout this paper we will use the convention that $C$ is an absolute constant which changes from line to line. 
By $L^p$ we mean the space of measurable functions $f$ on $\mathbb{R}^n$ so that $|f|^p$ is integrable, $L^\infty$ being the space of bounded functions on $\mathbb{R}^n$. By $W^{1,p}$ we mean the Sobolev space on $\mathbb{R}^n$ of $L^p$ functions whose derivative also belongs to $L^p$. We will also define the Lipschitz class using the following norm:
$$|f|_{\text{Lip}}= |f|_{L^\infty} + \sup_{x \not =y}\frac{|f(x)-f(y)|}{|x-y|}.$$
Moreover, we define the space $L^\infty([0,T]; \text{Lip})$ to be the set of all functions $f:\mathbb{R}^n\times\mathbb{R}\rightarrow \mathbb{R}$ for which $$|f|_{L^\infty ([0,T]; \text{Lip})}:=\sup_{t\in [0,T]} |f(t)|_{\text{Lip}}<\infty.$$Similar notation will be used for vector valued functions.

We will be making use of two kinds of commutators: a commutator of a singular integral operator and composition with a given function and a commutator of a singular integral operator and multiplication by a given function $[\cdot,\cdot]$ and $[\cdot,\cdot]_{*}.$

\begin{defn}

Let $n\geq 2$, $g\in L^\infty(\mathbb{R}^n),$ $\Phi:\mathbb{R}^n\rightarrow\mathbb{R}^n$ be a bi-Lipschitz measure preserving map, and $R$ be a Calder\'on-Zygmund singular integral operator with kernel $\frac{\Omega(\frac{x}{|x|})}{|x|^n}$ for some $\Omega$ which is smooth and mean-zero on $\mathbb{S}^{n-1}.$ Define the following bounded linear operators $[R,\Phi]$ and $[R,g]_{*}$ mapping $L^p$ to $L^p$ for $1<p<\infty$: 
$$[R,\Phi]f := R(f\circ\Phi)-R(f)\circ \Phi, $$
and
$$[R,g]_{*}f := R(fg)-R(f)g.$$
\end{defn}

\begin{rema}
The fact that $[R,\Phi]$ is bounded on $L^p$ does not actually require any regularity assumption on $\Phi$ if it is measure preserving, it only requires that $R$ is bounded from $L^p$ to $L^p$.
\end{rema}

We recall here the Littlewood-Paley decomposition. 
We define
$\mathcal{C}$ to be the ring of center
$0$, of small radius $1/2$ and great radius $2$. There exist two
nonnegative  radial
functions $\chi$ and $\phi$ belonging respectively to $C_0^\infty(B(0,1)) $ and to
$ C_0^\infty(\mathcal{C}) $ so that
\begin{equation}
\chi(\xi) + \sum_{q\geq 0} \phi (2^{-q}\xi) = 1,
\end{equation}
\begin{equation}
|p-q|\geq 2
\Rightarrow
\Supp\ \phi(2^{-q}\cdot)\cap \Supp\ \phi(2^{-p}\cdot)=\emptyset.
\end{equation}

\noindent
For instance, one can take $\chi \in C_{0}^\infty (B(0,1))$ such that $
\chi  \equiv 1 $ on $B(0,1/2)$ and take
$$
\phi(\xi) = \chi(\xi/2) -
\chi(\xi).
$$
Let us denote
by $\mathcal{F}$ the Fourier transform on $\R^d$. Let
$h,\
\wt h,\  \D_q,\  S_q$ ($q \in \Z$) be defined as follows:
%\noindent{\bf Notations}

$$h = { \mathcal{F}}^{-1}\phi\quad {\rm and}\quad \wt h =
{\mathcal F}^{-1}\chi,$$ $$
\D_qu = {\mathcal F}^{-1}(\phi(2^{-q}\xi)\mathcal F u) = 2^{qd}\int h(2^qy)u(x-y)dy, $$ $$
S_qu
=\mathcal{F}^{-1}(\chi(2^{-q}\xi)\mathcal{F} u) =2^{qd} \int \wt h(2^qy)u(x-y)dy.
$$

We point out that $S_q u = \sum_{q'\leq q-1} \D_{q'} u $. 
We define the inhomogeneous Besov spaces by

\begin{defi}
Let $s$ be a real number, p and r two real numbers greater than~$1$.
Then we define the following norm
$$
|u|_{ B^s_{p,r}} \equiv |S_0u|_{L^p}+\Big| \left (2^{qs} |\Delta
_qu|_{L^p}\right)_{q\in \N}\Big|_{\ell^r(\N)}.
$$
\end{defi}

\begin{defi}${\atop}$
Let $s$ be a real number, p and r two real numbers greater than~$1$.
We denote by~$ B^s_{p,r}$ the space of tempered distributions~$u$ such
that~$\|u\|_{ B^s_{p,r}}$ is finite.
\end{defi}
We refer to \cite{BCD} for the
proof of the following results:
\begin{lemm}
    $$  | \D_q u |_{L^b} \leq 2^{d({1 \over a} -{1 \over b})q } |\D_q u
|_{L^a}
\quad{\rm for }\ b\geq a \geq 1
$$
\end{lemm}
The following corollaries is straightforward:
\begin{cor}
{\sl
If~$b\geq a\geq 1$, then, we have the following continuous embeddings
\[
B^s_{a,r} \subset
B^{s-d\Bigl( \frac 1 a-\frac 1 b\Bigr)}_{b,r}.
\]
}
\end{cor}
\begin{cor}
$$B^{a}_{p,1} \subset L^\infty$$ if $ar=n$ and the imbedding is continuous. 
\end{cor}
\begin{rema}
Note that the importance of the spaces $B^{a}_{r,1}$ for our context is that singular integrals are generally bounded on these spaces and when $ar=n$, these spaces have the same scaling as $L^\infty$. 
\end{rema}
We recall that singular integral operators are bounded on Besov spaces. 
\begin{lemm} \cite{BCD} \label{CZBoundedness} If $R$ is a Calder\'{o}n-Zygmund singular integral operator with a kernel $K(x) = \frac{\Omega(\frac{x}{|x|})}{|x|^d},$ with $\Omega$ smooth and mean-zero on the unit-sphere, then $R$ is bounded from $B^{a}_{r,q}$ to itself for all $a\geq 0,r\geq 1,$ and $q\geq 1.$   
\end{lemm}

\addtocontents{toc}{\SkipTocEntry}

\subsection{The main (linear) result}

\begin{thm}     

\label{LowerBoundTheorem}

Let $n\geq 2$ and $u$ be a given divergence-free Lipschitz function, $u\in L^\infty([0,\delta];\text{Lip})$ for some $\delta>0$. Suppose that $R$ is a Calder\'{o}n-Zygmund singular integral operator in the sense of Lemma \ref{CZBoundedness}. 
Let $f_0\in L^1 \cap B^{\frac{1}{2}}_{2d,1}(\mathbb{R}^n)$  and let $f$ be the unique solution\footnote{See Proposition \ref{LWPBesov}.} of \eqref{LinearRTP1} with initial data $f_{0}.$
Then there exist two constants $C,c>0$ independent of $u$ and $f_0$  such that for all $0<t<\frac{c}{|u|_{L^\infty \text{Lip}}},$ we have 

\begin{equation} \label{LowerBound} |f(t)|_{L^\infty} \geq |tR(f_0)+f_{0}|_{L^\infty} -Ct^2( 1+ |u|_{L^\infty \text{Lip}} \exp(tC|u|_{L^\infty \text{Lip}}))|f_0|_{B^{\frac{1}{2}}_{2d,1}}. \end{equation}

\end{thm}

\begin{rema} Note that estimate \eqref{LowerBound} is only useful for $t$ small as the right-hand side of the estimate is negative for $t$ large. \end{rema}

The proof of Theorem \ref{LowerBoundTheorem} is based upon a non-trivial commutator estimate. 

\section{Proof of the Linear Theorem  \ref{LowerBoundTheorem}}

The subject of this section is the proof of Theorem \ref{LowerBoundTheorem}. The proof is based upon a non-trivial commutator estimate which we now present. 

\addtocontents{toc}{\SkipTocEntry}

\subsection{The Commutator Estimate}

\begin{prop} 
\label{CommutatorEstimate}
Let $n\geq 2.$ Let $\Phi$ be a bi-Lipschitz mapping from $\mathbb{R}^{n}$ to $\mathbb{R}^{n}$ which is measure preserving.   
Define the following commutator operating on $L^2 :$
$$[R,\Phi] \omega = R(\omega \circ \Phi) - R(\omega)\circ \Phi,  $$ where $R$ is as in Lemma 2.6. 
Let $0<a<1$ and $1<\rho<\infty.$ Then $[R,\Phi] : B^{a}_{\rho,1} \rightarrow B^{a}_{\rho,1}$ is bounded. Furthermore, there exists a universal constant $c$ depending only upon $R$ such that if 
$$M:=\max \{ |\Phi-Id|_{\text{Lip}},|\Phi^{-1}-Id|_{\text{Lip}} \}\leq c,$$
where $Id$ is the identity matrix, then,
\begin{equation} \label{Commutator} |[R,\Phi]\omega|_{B^{a}_{p,1}} \leq C M |\omega|_{B^{a}_{p,1}}, \end{equation}
with the constant $C$ depending only upon $|\Phi|_{Lip},$ $|\Phi^{-1}|_{Lip},$ the dimension $n,$ and the operator $R$.
\end{prop}

We first would like to relate this proposition to a well-known problem in harmonic analysis, which is the boundedness of the Calder\'{o}n commutators. Indeed, if we write $$ [R,\Phi]\omega\circ \Phi^{-1}= R(\omega \circ \Phi)\circ \Phi^{-1} - R(\omega) = K \star (\omega\circ \Phi) \circ \Phi^{-1} - K\star \omega. $$ Now, using the fact that $\Phi$ is measure preserving, we get   
$$ [R,\Phi]\omega\circ \Phi^{-1} (x)= \int_{\mathbb{R}^{d}} \Big[K(\Phi^{-1} (x)-\Phi^{-1} (y)) -K(x-y)\Big] \omega(y)dy.  $$
In the special case of the Hilbert transform, we get $$[H,\Phi]\omega \circ\Phi^{-1} (x) =\int_{\mathbb{R}}\Big[ \frac{1}{\Phi^{-1} (x)-\Phi^{-1} (y)} -\frac{1}{x-y} \Big]\omega(y) dy $$
$$ =\int_{\mathbb{R}}[ \frac{x-y}{\Phi^{-1} (x)-\Phi^{-1} (y)} -1 ]\frac{\omega(y)}{x-y} dy $$
$$=\int_{\mathbb{R}} \frac{(I-\Phi^{-1})(x) - (I-\Phi^{-1})(y)}{x-y} \cdot \frac{x-y}{\Phi^{-1} (x)-\Phi^{-1} (y)} \cdot \frac{\omega(y)}{x-y}dy. $$ 
Thus, we see that in order to estimate this commutator, we would need to estimate operators of the form:
$$T_{A}(\omega) = \int_{\mathbb{R}} F (\frac{A(x)-A(y)}{x-y}) \frac{\omega(y)}{x-y}dy,$$ where $A$ is a Lipschitz function. Estimates of this type have been studied in $L^p$ spaces by many authors. We refer the reader to the recent book of Muscalu and Schlag \cite{MuscaluSchlag} and the references therein. Fortunately, due to the large literature on these operators, we will be able to use some of the existing results to prove estimate \eqref{Commutator}.   In particular, in \cite{Murray}, Murray proved the lemma in the $L^2$ case. Here, we show how the modern theory of Calder\'on-Zygmund operators is used to extend her result to the $L^p$ case. This is, admittedly, a simple exercise for the specialists.

We recall here a simple consequence of Murray's theorem (Theorems 6.1 and 6.2 of \cite{Murray}):
\begin{thm} \label{Murray}(Murray, 1985) Proposition \ref{CommutatorEstimate} holds with $B^{a}_{p,1}$ replaced by $L^2.$
\end{thm}
In fact, Murray proved that the mapping $\Phi\rightarrow \Phi R\Phi^{-1}$ is \emph{analytic} for $\Phi$ in a neighborhood of the identity. We will use Theorem \ref{Murray} to prove Proposition \ref{CommutatorEstimate}. The strategy is as follows:

(1) Use Theorem \ref{Murray} and pass from the commutator estimate on $L^2$ to the same estimate on $L^p$; here we will have to be careful not to lose the $|\Phi-Id|_{Lip}$ factor when passing from the $L^2$ to the $L^p$ estimate.

(2) Use the $L^p$ estimate to prove a $W^{1,p}$ estimate. Here we will make use of the classical Coifman-Rochberg-Weiss commutator estimate, which will allow us to commute \emph{multiplication} by a bounded function and a Riesz operator. 

(3) Conclude using the method of real interpolation. 

\begin{proof}[Proof of Proposition 3.1]
The proof relies upon first observing that $[R,\Phi]$ is a \emph{linear} operator and thus, using results from the theory of interpolation, it suffices to show that $[R,\phi]$ satisfies estimate \eqref{Commutator} on $L^p$ and $W^{1,p}.$  Namely, it suffices to prove the following two inequalities:

$$|[R,\Phi]\omega|_{L^p} \lesssim M |\omega|_{L^p}$$
and
$$|[R,\Phi]\omega|_{W^{1,p}} \lesssim M|\omega|_{W^{1,p}},$$ for all $1<p<\infty.$

To pass from $L^2$ to $L^p,$ the work will go into writing our commutator in the form of a so-called standard kernel and proving various estimates on the kernel. As for the $W^{1,p}$ estimate, we show that this is a consequence of the $L^p$ estimate. 

\addtocontents{toc}{\SkipTocEntry}

\subsubsection{The $L^p$ estimate}

To get the $L^p$ estimate, we are first going to show that the commutator can be written as an operator with a so-called standard kernel \cite{Grafakos}. 
Indeed, define $T$ as follows:
\begin{equation} T(\omega):=[R,\Phi^{-1}]\omega\circ \Phi (x)= \int_{\mathbb{R}^{d}} J(x,y) \omega(y)dy,  \end{equation}
with $J(x,y):=K(\Phi (x)-\Phi (y)) -K(x-y).$\footnote{Note we have interchanged $\Phi^{-1}$ and $\Phi$ just for notational convenience; they play more or less the same role throughout the paper.}
Note that since $\Phi$ is measure preserving, $L^p$ estimates on $T$ and $L^p$ estimates on $[R,\Phi^{-1}]$ are equivalent. 

\vspace{3mm}

\noindent\emph{ Claim:}
\begin{equation} \label{JEstimate1} |J(x,y)|\leq C \frac{M}{|x-y|^n},\end{equation}
\begin{equation} \label{JEstimate2} |\nabla J(x,y)|\leq C \frac{M}{|x-y|^{n+1}},\end{equation}
where $M=\max \{|\Phi-Id|_{Lip},|\Phi^{-1}-Id|_{Lip} \}$ and with $C$ depending only on the dimension and on the Lipschitz norm of $\Phi$ and its inverse. 
\vspace{3mm}

\noindent\emph{Proof of Claim:}

\vspace{3mm}

\noindent Recall that $$K(x)=\frac{\rho(x)}{|x|^n}$$ where $\rho$ is smooth on the unit sphere. 
Therefore, $$J(x,y)=\frac{\rho(x-y)|\Phi(x)-\Phi(y)|^n-\rho(\Phi(x)-\Phi(y))|x-y|^n}{|x-y|^n|\Phi(x)-\Phi(y)|^n} $$
Let's define $$\Psi:=\Phi-Id.$$
Then, 
$$J(x,y)= \frac{\rho(x-y)|\Psi(x)-\Psi(y)-x-y|^n-\rho(\Phi(x)-\Phi(y))|x-y|^n}{|x-y|^n|\Phi(x)-\Phi(y)|^n}$$
$$=\frac{\rho(x-y)(|\Psi(x)-\Psi(y)-x-y|^n-|x-y|^n)+|x-y|^n(\rho(x-y)-\rho(\Phi(x)-\Phi(y)))}{|x-y|^n|\Phi(x)-\Phi(y)|^n}$$
$$:=I+II,$$ with
$$I:=\frac{\rho(x-y)\Big(|\Psi(x)-\Psi(y)-x-y|^n-|x-y|^n\Big)}{|x-y|^n|\Phi(x)-\Phi(y)|^n},$$
and $$II:=\frac{\Big (\rho(x-y)-\rho(x-y+\Psi(x)-\Psi(y))\Big )|x-y|^n}{|x-y|^n|\Phi(x)-\Phi(y)|^n}.$$
The fact that $I$ satisfies \eqref{JEstimate1}-\eqref{JEstimate2} is clear because:
$$|\Psi|_{Lip}\leq M, \quad\frac{|\Phi(x)-\Phi(y)|}{|x-y|}\lesssim 1,\,\, \text{and} \quad \frac{|x-y|}{|\Phi(x)-\Phi(y)|}\lesssim 1.$$
$II$ is similarly controlled because $\rho$ is smooth on the unit sphere. Here, we have used the following inequalities:

$$|\rho(a+b)-\rho(a)|\leq C |b|$$ and $$|a+b|^n-|a|^n\leq C(|a|^{n-1}|b|+|b|^n),$$ for a fixed constant $C$ only depending on $\rho.$
This completes the proof of the claim. 

Now that $J$ satisfies \eqref{JEstimate1}-\eqref{JEstimate2}, and since, by Theorem \ref{Murray}, an $L^2$ commutator estimate holds:
$$[R,\Phi]_{L^2\rightarrow L^2}\lesssim M,$$
we are in a position to pass to the $L^p$ estimate. Recall the following standard theorem \cite{Grafakos}:

\begin{prop}
Assume $J(x,y)$ satisfies 
$$|J(x,y)|\lesssim \frac{A}{|x-y|^n}$$
and 
$$|\nabla J(x,y)|\leq \frac{A}{|x-y|^{n+1}}.$$ Let $T$ be the singular integral operator associated with the kernel $J.$ Assume that
$$|T|_{L^2\rightarrow L^2}\leq B.$$
Then, $T$ maps $L^p$ to itself for $1<p<\infty.$ 
Moreover, 
$$|T|_{L^p\rightarrow L^p}\leq C(n)\max\{p,(p-1)^{-1}\}(A+B)$$
\end{prop}

Using \eqref{JEstimate1}-\eqref{JEstimate2}, Proposition 3.3, as well as Theorem \ref{Murray}, we conclude:
$$[R,\Phi]_{L^p\rightarrow L^p}\leq C(|\Phi|_{Lip},|\Phi^{-1}|_{Lip},n)\max \{p, (p-1)^{-1} \} M,$$
where $$M=\max \{|\Phi-Id|_{Lip},|\Phi^{-1}-Id|_{Lip} \}.$$

\addtocontents{toc}{\SkipTocEntry}

\subsubsection{The $W^{1,p}$ estimate}

Consider the commutator applied to some $\omega \in W^{1,p},$ $[R,\Phi]\omega.$ 
In order to prove the desired estimate, we need to estimate $\partial_{x} ([R,\Phi]\omega)$ in $L^p.$

$$\partial_{x} ([R,\Phi]\omega) = \partial_{x} R(\omega \circ \Phi)- \partial_{x} (R(\omega) \circ \Phi).$$
Note that $\partial_{x}$ and $R$ commute. 

Now we compute: $$ \partial_{x} R(\omega \circ \Phi)- \partial_{x} (R(\omega) \circ \Phi)= R(\nabla \omega \circ \Phi \cdot \partial_{x} \Phi)-R(\nabla \omega) \circ \Phi \cdot \partial_{x} \Phi$$
$$=R(\nabla \omega \circ \Phi \cdot \partial_{x} \Phi)- R(\nabla \omega \circ \Phi) \cdot \partial_{x} \Phi$$ $$ +R(\nabla \omega \circ \Phi) \cdot \partial_{x} \Phi-R(\nabla \omega) \circ \Phi \cdot \partial_{x} \Phi $$
$$=[R,\partial_{x}\Phi]_{*}\nabla\omega\circ \Phi+[R,\Phi] \nabla \omega \cdot \partial_{x} \Phi,$$ where we recall that
$$[R,A]_{*} B = R(A\cdot B)-R(A)\cdot B.$$
Thus, $$\partial_{x} ([R,\Phi]\omega)= [R,\partial_{x}\Phi]_{*}\nabla\omega\circ \Phi+\Big([R,\Phi] \nabla \omega\Big) \cdot \partial_{x} \Phi.$$
$$=[R,\partial_{x}(\Phi-I)]_{*}\nabla\omega\circ \Phi+\Big([R,\Phi] \nabla \omega\Big) \cdot \partial_{x} \Phi.$$ Note that we can subtract the identity mapping from $\Phi$ in the first term without changing anything because the identity commutes with multiplication (as opposed to the second commutator which is a commutator with composition).  

To estimate the first term, we use the Coifman-Rochberg-Weiss \cite{Grafakos} commutator estimate and to estimate the second we use the $L^p$ estimate from above. Thus,  $$|\partial_{x} ([R,\Phi]\omega)|_{L^p} \lesssim |\nabla \omega|_{L^p} ( |\Phi-I|_{Lip}+ |\Phi|_{Lip} |I-\Phi|_{Lip}).$$
This concludes the $W^{1,p}$ estimate. Now that we have the $L^p$ and the $W^{1,p}$ estimate, we can use the method of real interpolation to conclude the corresponding $B^{a}_{p,1},$ estimate for all $0<a<1$. This concludes the proof of Proposition 3.1.

\end{proof}

\addtocontents{toc}{\SkipTocEntry}

\subsection{Estimates on the flow}

Given a Lipschitz velocity field $u$ we may solve the following ordinary differential equation, to find the flow induced by $u$:
$$\dot{\Phi}(x,t)= u(\Phi(t,x),t),$$
$$\Phi(x,0)=x.$$
In the following, we will prove bounds on the size of the (time dependent) Lipschitz norm of $\Phi$ and $\Phi-Id.$ To make the notation simpler, we will write just Lip for $L^\infty([0,t]; \text{Lip})$.
Because $u$ is divergence free, $\Phi$ is measure preserving. Now, we may write $$\Phi(x,t) = x +\int_{0}^{t} u(\Phi(x,\tau),\tau)d\tau.$$
Thus, $$\Phi(\cdot,t) - I = \int_{0}^{t} u(\Phi(\cdot,\tau),\tau)d\tau.$$
First, by Gronwall's lemma\footnote{To prove Lipschitz estimates on $\Phi$ and $\Phi^{-1}$ we notice that for all $x,y\in\mathbb{R}^n$ $$|x-y|- \int_0^t |\nabla u|_{L^\infty} |\Phi(x,\tau)-\Phi(y,\tau)| d\tau\leq |\Phi(x)-\Phi(y)| \leq |x-y|+ \int_0^t |\nabla u|_{L^\infty} |\Phi(x,\tau)-\Phi(y,\tau)|d\tau.$$ Hence Gronwall's lemma gives us that $$e^{-t|\nabla u|_{L^\infty}}\leq \frac{|\Phi(x)-\Phi(y)|}{|x-y|}\leq e^{t |\nabla u|_{L^\infty}}$$ which gives the Lipschitz bound on $\Phi$ and $\Phi^{-1}$.}, $$|\Phi|_{\text{Lip}} \leq \exp (t | u|_{\text{Lip}}).$$
Consequently, due to the fact that $|u\circ \Phi|_{\text{Lip}}\leq |u|_{\text{Lip}}|\Phi|_{\text{Lip}},$ $$|\Phi-I|_{\text{Lip}} \leq t | u|_{\text{Lip}}|\Phi|_{\text{Lip}}.$$ and similarly for $\Phi^{-1} (\cdot,t).$
In particular, 
\begin{equation}\label{FlowControl} |\Phi-I|_{\text{Lip}} \leq t | u|_{\text{Lip}}\exp(t|u|_{\text{Lip}}). \end{equation}

\addtocontents{toc}{\SkipTocEntry}

\subsection{Local well-posedness in the critical besov space}

Because the velocity field is Lipschitz, we get that the transport equation is well-posed in all the Besov spaces $B^{a}_{\rho,1}$, $0<a<1, \rho<\infty.$ In particular, the transport equation is locally well-posed in the critical Besov spaces $B^{\frac{d}{p}}_{p,1}$ for all $ d<p<\infty$  (note that these spaces imbed in $L^\infty$).

In particular, we have the following proposition.

\begin{prop}
\label{LWPBesov}
Let $u \in L^\infty ( [0,1], Lip)$ and let $R$ be a Calder\'{o}n-Zygmund singular integral operator. Then (2.1)-(2.3) is well-posed in $B^{a}_{\rho,1}$ for every $a\in[0,1]$ and every $\rho\in[1,\infty]$ in the sense that if $f_0 \in B^{a}_{\rho,1}$ then there exists a unique solution $f\in C([0,1]B^{a}_{\rho,1})$ which solves (2.1)-(2.2) with initial condition $f_0.$  Moreover, \begin{equation} \label{LWPBesovBound} |f(t)|_{B^{a}_{\rho,1}} \leq |f_0|_{B^{a}_{\rho,1}} \exp(tC|\nabla u|_{L^\infty Lip}) \end{equation} for all $t\in [0,1].$ 

\end{prop}

\addtocontents{toc}{\SkipTocEntry}

\subsection{The equation along the flow}

Recall that $$f_{t} +u\cdot \nabla f =R(f).$$ Then consider the flow map $\Phi$ as in Subsection 3.2. Hence, $f$ satisfies the following equation: $$(f\circ \Phi)_{t} = R(f) \circ \Phi.$$
In particular, $$(f\circ \Phi)_{t} = R(f\circ \Phi) + [R,\Phi] f.$$
Thus, by Duhamel's principle, $$f\circ \Phi =\exp(Rt) f_{0} + \int_{0}^{t} \exp(R(t-s))[R,\Phi]f(\tau)d\tau.$$
In particular, $$|f|_{L^\infty} \geq |\exp(Rt) f_0|_{L^\infty} - C\int^{t}_{0} \Big|\exp(R(t-s))[R,\Phi]f(\tau)\Big|_{B^{\frac{1}{2}}_{2d,1}}d\tau,$$ where we have used that $ B^{\frac{1}{2}}_{2d,1} \hookrightarrow L^\infty.$
Observe that $R$ is bounded on $B^{\frac{1}{2}}_{2d,1}$ and $t\in [0,1].$ Thus,  
$$|f|_{L^\infty} \geq |\exp(Rt) f_0|_{L^\infty} - C\int^{t}_{0} \Big|[R,\Phi]f(\tau)\Big|_{B^{\frac{1}{2}}_{2d,1}}d\tau.$$
Now note that $$\exp(Rt)=  I +tR +t^2\sum_{n=2}^{\infty} \frac{t^{n-2}R^n}{n!}.$$ and thus $$|\exp(Rt)f_0 |_{L^\infty} \geq |tR(f_0)+f_{0}|_{L^\infty} -t^2 C|f|_{B^{\frac{1}{2}}_{2d,1}}.$$ 
In particular, $$|f|_{L^\infty} \geq |tR(f_0)+f_{0}|_{L^\infty} -t^2 C|f|_{B^{\frac{1}{2}}_{2d,1}}-C\int^{t}_{0} |[R,\Phi]f(\tau)|_{B^{\frac{1}{2}}_{2d,1}}d\tau. $$

We now use the commutator estimate \eqref{Commutator} as well as estimate \eqref{FlowControl}. Note that for $t$ small enough, $\Phi$ and its inverse will be arbitrarily close to the identity. In particular,  

$$|f|_{L^\infty} \geq |tR(f_0)+f_{0}|_{L^\infty} -t^2 C|f_0|_{B^{\frac{1}{2}}_{2d,1}}-Ct^2 |u|_{L^\infty \text{Lip}} \exp(t|u|_{L^\infty \text{Lip}})\sup_{\tau \in [0,1]}|f(\tau)|_{B^{\frac{1}{2}}_{2d,1}}. $$
Then we use the local well-posedness in the critical Besov space, estimate \eqref{LWPBesovBound} and we see that 
$$|f|_{L^\infty} \geq |tR(f_0)+f_{0}|_{L^\infty} -t^2 C|f_0|_{B^{\frac{1}{2}}_{2d,1}}-Ct^2 |u|_{L^\infty \text{Lip}} \exp(tC|u|_{L^\infty \text{Lip}})|f_0|_{B^{\frac{1}{2}}_{2d,1}}. $$

This concludes the proof of Theorem \ref{LowerBoundTheorem}. 

\qed

\section{General Application of the Linear Estimate}

In this section we give a very mild condition on $R$ which ensures that the linear system \ref{LinearRTP1} is strongly ill-posed in $L^\infty(\mathbb{R}^n)$ in the sense of Definition 1.1.

\begin{assumption}\label{Assumption1}
There exists a sequence of functions $g_{N} \in {B^{a}_{\rho,1}} $ such that the following holds:

\begin{equation} |g_{N}|_{L^\infty} \leq 1,\end{equation}
\begin{equation} |R(g_{N})|_{L^\infty} \geq cN,\end{equation}
\begin{equation} |R(g_{N})|_{B^{a}_{\rho,1}} \leq CN, \end{equation}

where $c$ and $C$ are constants independent of $N,$ $0<a<1,$  $ 1<\rho<\infty,$ and $a\rho=n$ ($n$ is the dimension). 

\end{assumption}

Heuristically, Assumption 1 says that there exists an $L^\infty$ function $g$ such that $R(g)$ has a logarithmic singularity. Assumption 1 can be shown to hold for many singular integral operators such as the Hilbert transform, the Riesz transforms (and compositions of Riesz transforms), and others. 

Now we can state the (linear) ill-posedness theorem. 

\begin{thm}
If $R$ satisfies Assumption 1 and if $u \in L^\infty Lip$ then (2.1)-(2.3) is strongly ill-posed on $L^\infty$. 
\end{thm}

\begin{proof}
The proof is a direct consequence of Theorem \ref{LowerBoundTheorem}. By Assumption 1, there exists $g_{N}$ satisfying (4.1)-(4.3). Fix $\epsilon>0.$ One can then solve (2.1)-(2.3) in $B^{a}_{\rho,1}$ with initial data $\epsilon g_{N}$ and call the solution $f_{N}(t)$. Theorem \ref{LowerBoundTheorem} now applies so that:
$$ |f_N(t)|_{L^\infty} \geq \epsilon \Big  ( t|R(g_N)|_{L^\infty}-|g_N|_{L^\infty} -Ct^2 \Big ( 1+ |u|_{L^\infty \text{Lip}} \exp(tC|u|_{L^\infty \text{Lip}})\Big)|g_N|_{B^{\frac{1}{2}}_{2d,1}}\Big ).$$
Using (4.1)-(4.3), we see:
$$|f_{N} (t)|_{L^\infty} \geq \epsilon tcN- \epsilon- \epsilon Ct^2 (1 +|u|_{L^\infty \text{Lip}} \exp(tC|u|_{L^\infty \text{Lip}}))N.$$
First we take $t = \frac{\alpha}{1+|u|_{L^\infty Lip}}$ for some small enough constant $\alpha$ so that $$|f_{N}(t)|_{L^\infty} \geq \frac{\epsilon \alpha^2 N}{1+|u|_{L^\infty Lip}}.$$
Then we take $N=\frac{1+|u|_{L^\infty Lip}}{\alpha^2\epsilon^2}.$
Hence, $|f_{N}(t)|_{L^\infty} \geq \frac{1}{\epsilon}$ even though $|f_{N} (0)|_{L^\infty} \leq \epsilon.$
This concludes the proof.

\end{proof}

In the coming sections, the proof of Theorem 4.2 will be used to show mild ill-posedness for some \emph{non-linear} equations.

\section{Perturbations of the 2D Euler equations}

An interesting open problem in mathematical fluid dynamics is to prove global well-posedness for the following type of equation:
\begin{equation} u_t + (u \cdot \nabla) u + \nabla p = Au \end{equation}
\begin{equation} div(u)=0, \end{equation}
where $A$ is some constant matrix. It is possible to prove global well-posedness in only one case: when $Au = \lambda u + \gamma u^\perp,$ for constants $\lambda$ and $\gamma$ (of course $A$ can be taken to depend on $x$ in a similar fashion). Indeed, in this case $\text{curl} (Au)=\lambda\omega$ and thus $\omega$ satisfies a maximum principle which leads to global well-posedness using the standard technique. 

When $A$ is not of the above form, we will use Theorem 2.1 to prove a mild ill-posedness result for (5.1)-(5.2) (in other words, to show that $\omega$ \emph{does not} satisfy a maximum principle). The reason that we will not be able to prove the strong ill-posedness is that the non-linear term starts to play a prominent role once the vorticity grows. 
 As a special case, we consider the following system:
 \begin{equation} u_t + (u \cdot \nabla) u + \nabla p = \left( \begin{array}{cc}
-u_1 \\
0 \\
\end{array} \right), \end{equation}
\begin{equation} \text{div}(u)=0. \end{equation}
Notice that the right-hand side of this equation is a drag term--it causes the energy of the system to decrease. A simple computation shows that $|\omega|_{L^2}$ is also decreasing. Therefore, on the level of kinetic energy, we should expect this system to behave "better" than 2-D Euler. It turns out that this drag term destroys the conventional global well-posedness proof for 2-D Euler as well as the Yudovich theory. 
Upon passing to the equation for the vorticity we get:
\begin{equation} \omega_t +u\cdot \nabla \omega=-{u_1}_y, \end{equation}
\begin{equation} u=\nabla^{\perp}(-\Delta)^{-1}\omega. \end{equation}
In particular, using (5.6), $$-{u_{1}}_y =R_{2}^2 \omega,$$ where $R_{2}$ is the Riesz transform with symbol $\frac{-i\xi_{2}}{|\xi|}.$

Using Theorem \ref{LowerBoundTheorem}, we prove the following non-linear ill-posedness result for this system.

\begin{thm} \label{RTP1}
(5.5)-(5.6) is mildly ill-posed in $L^\infty.$ In other words, 
there exists a sequence of functions $\omega_{0}^{\epsilon}$ belonging to $H^s$ for every $s>0$ and universal constants $C_{i},$ independent of $\epsilon,$ with the following properties:
\\
$(1) |\omega_{0}^\epsilon|_{L^\infty} \leq \epsilon,$ 
\\
$(2) |\omega_{0}^\epsilon|_{B^{\frac{1}{2}}_{4,1}} \leq C_{1},$
\\
$(3)$ If $\omega^\epsilon(t)$ is the (local) solution of (5.5)-(5.6) in $L^{\infty}([0,C_{2}];B^{\frac{1}{2}}_{4,1})$ with $\omega^\epsilon (0) = \omega_{0}^\epsilon,$
then there exists some $t\in (0,\epsilon]$ so that
$$|\omega(t)^\epsilon|_{L^\infty} \geq C_{3}.$$
\end{thm}
\begin{rema} Note that we may also take $t<\delta$ independent of $\epsilon.$ \end{rema}
\begin{rema} We conjecture that the system is actually strongly ill-posed. However, because of the fact that the Lipschitz norm of $u$ may become unbounded, it is unclear whether the non-linear term could make the solution bounded in $L^\infty$ as is the case in the example of Section 2.2.\end{rema}

\begin{rema} Note here that our result holds for more general equations of the following type

$$\omega_{t} + u\cdot \nabla \omega = R \omega $$

$$u = (-\Delta)^{-\alpha} \nabla^\perp \omega,$$ 

where $\alpha \geq 1$ and $R$ is a linear operator mapping $B^{a}_{\rho,1}$ to itself and for which there exists some $\omega_{0} \in L^{\infty}$ such that $R\omega_{0}$ has a logarithmic singularity. Therefore any Calderon-Zygmund operator which is unbounded on $L^\infty$ would work. Due to the extra regularity on the velocity field, it is easy to show that for $\alpha>1,$ the system is actually \emph{strongly} ill-posed in $L^\infty$.    
\end{rema}
We have a few obvious corollaries. 

\begin{cor}

The zero solution of (5.5)-(5.6) is (non-linearly) unstable with respect to $L^{\infty}$ perturbations. 

\end{cor}

\begin{cor}
The map $J_{t}$ taking an initial data $\omega_{0}$ to the solution at time $\omega (t)$is discontinuous in the $L^{\infty}$ norm.
\end{cor}

To prove Theorem 5.1 we first need to prove that $R_{2}^2$ satisfies Assumption 1.

\addtocontents{toc}{\SkipTocEntry}

\subsection{Proof of Theorem 5.1}

The proof of theorem 5.1 is based upon the linear Theorem \ref{LowerBoundTheorem}. Indeed, suppose that $R:=R_{2}^2$ satisfies Assumption 1. Note that, following a result of Vishik \cite{V2}, one can prove local well-posedness of (5.3)-(5.5) in all spaces $B^{a}_{\rho,1}$ with $a\rho= 1$ (in fact this is a consequence of Proposition \ref{LWPBesov}). Indeed, the following is standard:

\begin{lemm}\label{LWP}

Let $R$ be a Calder\'on-Zygmund operator. Consider the following equation in the plane:
\begin{equation} \omega_{t} + u \cdot\nabla \omega = R\omega, \end{equation}
\begin{equation} u= (-\Delta)^{-1} \nabla^{\perp} \omega. \end{equation}
Then, this system is locally well-posed in the Besov space $B^{a}_{\rho,1}$ and the following estimate holds for all $t$ with $Ct|\omega_0|_{B_{\rho,1}^a}<\frac{1}{2}$:
\begin{equation} \label{LWPRTPBesov} |\omega(t)|_{B^{a}_{\rho,1}} \leq C\frac{|\omega_{0}|_{B^{a}_{\rho,1}}}{1-Ct|\omega_{0}|_{B^{a}_{\rho,1}}}. \end{equation} 
\end{lemm}

By Lemma \ref{Assumption1_R_2} below, there exists $g_{N}$ satisfying:
$$ |g_{N}|_{B^{\frac{1}{2}}_{4,1}} \leq CN,$$ 
$$ |g_{N}|_{L^\infty} \leq 1,$$
$$ |Rg_{N}|_{L^{\infty}} \geq cN.$$

Now solve system (5.5)-(5.6) with initial data $\epsilon g_{N}.$ Note that in what follows, $\epsilon tN$ will always be smaller than some fixed constant $c.$ This will ensure that we have existence on a uniform time interval. 

Call the solution $f_{N}(t) $  and its corresponding velocity field $u_N(t).$ Following the ideas from Section 4, we see that using the linear estimate \eqref{LowerBound}, we get
$$ |f_N(t)|_{L^\infty} \geq \epsilon (t|R(g_N)|_{L^\infty}-|g_N|_{L^\infty} -Ct^2( 1+ |u_N|_{L^\infty \text{Lip}} \exp(tC|u_N|_{L^\infty \text{Lip}}))|f_N|_{L^\infty B^{\frac{1}{2}}_{2d,1}}).$$
We then see, due to the local well-posedness of (5.5)-(5.6), that
$$|f_{N} (t)|_{L^\infty} \geq \epsilon tcN- \epsilon- \epsilon Ct^2 (1 +|u_N|_{L^\infty \text{Lip}} \exp(tC|u_N|_{L^\infty \text{Lip}}))N.$$
Using that $|\nabla u_N|_{L^\infty} \lesssim |\omega|_{B^{\frac{1}{2}}_{4,1}} \lesssim \epsilon N,  $ we get
$$|f_{N} (t)|_{L^\infty} \geq \epsilon tcN- \epsilon- \epsilon Ct^2 (1 +\epsilon N \exp(tC\epsilon N))N.$$
Now if we take $\epsilon N t $ small (but independent of $\epsilon$), we see that for $\epsilon$ and $t$ small enough 
$$|f_{N}(t)|\geq c\epsilon Nt- C\epsilon^2 N^2 t^2.$$
Upon taking $\epsilon N t$ smaller yet (on the order of $\frac{c}{2C}$) we see that $$|f_{N}(t)|\geq\alpha,$$ for some absolute constant $\alpha.$

We are done once we note that $|f(0)|_{L^\infty}\leq \epsilon$. 

\qed

\begin{rema}The reason that we are unable to prove the strong ill-posedness for (5.7)-(5.8) is that once the vorticity becomes large, the commutator estimate we have becomes uncontrollable. If there were a way to control the non-linear term by something less than $\Phi$ in the Lipschitz class (say, if one were able to do with only a $C^\alpha$ bound on $\Phi$), then the strong ill-posedness would be within reach. This is a challenge. 
\end{rema}
\addtocontents{toc}{\SkipTocEntry}

\subsection{Proof that $R_{2}^2$ satsisfies Assumption 1}

\begin{lemm}
\label{Assumption1_R_2}
$R:=R_{2}^2$ satisfies Assumption 1.
\end{lemm}

\vspace{2mm} 

\noindent\emph{Proof of the Lemma:}

\vspace{3mm}

By a rotation, it suffices to prove that $R_1R_2$ satisfies Assumption 1 since, under rotation by $\frac{\pi}{2}$, $R_2^2$ becomes $2R_1R_2-Id.$ 

To show that $R_1R_2$ satisfies Assumption 1, we define $f^{N}$ on the fourier side by $\widehat{f_{N}}= \chi_{[-2^N,2^N]^2} \hat{f}$, where $f(x,y)=\chi_{[-1,1]^2} \text{sgn}(x)\text{sgn}(y).$ Note that this is a regularization of the stationary solution of the Euler equations used in the work of Bahouri and Chemin \cite{BC94}. Then, clearly $f_{N}$ belongs to $H^{s}$ for all $s$.  

First note that $$\hat{f} (\xi_{1},\xi_{2})= 4 \frac{\sin(\xi_{1})\sin(\xi_{2})}{\xi_{1}\xi_{2}}.$$
\vspace{2mm}
\emph{Proof that $f_{N}$ satisfies condition (1):}
\vspace{2mm}

Note that $f$ belongs to $B^{\frac{1}{2}}_{4,\infty}.$ Indeed, for $|\xi|$ large, $\widehat{(-\Delta)^{\frac{1}{2}}f}(\xi)$ is a smooth function multiplied by $\frac{1}{|\xi|^{3/2}}$. Showing that $f$ belongs to $B^{\frac{1}{2}}_{4,\infty}$ is then an exercise (see for example Proposition 2.21 of \cite{BCD}).   

Then, $$|\chi_{B_{2^N}} f|_{B^{\frac{1}{2}}_{4,1}} \leq \sum_{1}^{CN}|\Delta_{j}f|_{B^{\frac{1}{2}}_{4,\infty}}+|S_0f|_{L^4}.$$ 
  
This implies condition (1).

\vspace{2mm}

\emph{Proof that $f_{N}$ satisfies condition (2):}

\vspace{2mm}

By the Fourier inversion formula, we have that 
$$|f_N|_{L^\infty} \leq \sup_{x_{1},x_{2}} \int_{-2^N}^{2^N}\int_{-2^N}^{2^N} \frac{\sin(\xi_{1})\sin(\xi_{2})}{\xi_{1}\xi_{2}}\cos(x_1\xi_{1})\cos(x_2\xi_{2})d\xi_{1}d\xi_{2}  $$
To prove condition (2) it suffices to show that the following quantity is bounded:
$$\sup_{x} \int_{-2^N}^{2^N} \frac{\sin(\xi)}{\xi}\cos(x\xi)d\xi=\sup_{x} \int_{-2^N}^{2^N} \frac{\sin(\xi+x\xi)-\sin(\xi-x\xi)}{2\xi}d\xi.$$
$$=\sup_x \int_{-2^N(x+1)}^{2^N(x+1)} \frac{\sin(\xi)}{2\xi}d\xi+ \int_{-2^N(-x+1)}^{2^N(-x+1)} \frac{\sin(\xi)}{2\xi}d\xi.$$
which is bounded by a universal constant since the following quantity is known to be bounded:

$$\sup_{a,b} |\int_{a}^{b} \frac{\sin(\xi)}{\xi}d\xi|<C.$$

\emph{Proof that $f_{N}$ satisfies condition (3):}
Using the Fourier inversion formula, we see that 

$$R_1R_2f_N(x,y)= \int_{[-2^N,2^N]^2} \frac{\sin(\xi_1)\sin(\xi_{2})}{\xi_{1}^2 +\xi_{2}^2}\sin(x\xi_{1})\sin(y\xi_{2})d\xi_{1} d\xi_{2}.$$
$$R_1R_2f_N(1,1)=\int_{[-2^N,2^N]^2} \frac{\sin^2 (\xi_1)\sin^2 (\xi_{2})}{\xi_{1}^2 +\xi_{2}^2}d\xi_{1}d\xi_{2}$$ and condition (3) follows. 
\qed

\section{The 3d Euler equations}\label{3dEuler}

Consider the three dimensional vorticity equation:
\begin{equation} \omega_{t} +u\cdot \nabla \omega = \nabla u \, \omega. \end{equation}
It is not clear at first that the 3d Euler equations can be cast in the framework of the linear problem (2.1)-(2.3). As above, through the Biot-Savart law, one can view $\nabla u$ as $R(\omega)$ where $R$ is now a matrix of singular integral operators.So the 3d Euler equations can be seen as:
\begin{equation} \omega_{t} +u\cdot \nabla \omega =R(\omega)\omega. \end{equation} 
The quadratic nature of $R(\omega)\omega$ is such that we cannot directly apply the analysis of (2.1)-(2.3). However, one can consider perturbing a shear flow in order to pull a linear $R(\omega)$ out of the right hand side. Indeed, let $\omega ={\bf e_{3}}+\tilde{\omega}.$

Then we see that $$\tilde{\omega}_t +u\cdot\nabla \tilde\omega=R(\tilde\omega)\tilde\omega+R(\bf{e_3})\tilde\omega+R(\tilde \omega)\bf{e_{3}}.  $$
Note that we may regard $R(\tilde\omega)\tilde\omega$ as a quadratic term so that it is of order $\epsilon^2$ if we follow the proof of Theorem 5.1. We can deduce the following theorem:

\begin{thm} There exists a sequence of functions $\omega_{0}^{\epsilon}$ belonging to $H^s$ for every $s>0$ and universal constants $C_{i},$ $1\leq i\leq 3$, independent of $\epsilon,$ with the following properties:
\\
$(1) |\omega_{0}^\epsilon-{\bf{e_3}}|_{L^\infty} \leq \epsilon$ 
\\
$(2) |\omega_{0}^\epsilon-{\bf{e_3}}|_{B^{\frac{3}{4}}_{4,1}} \leq C_{1}$
\\
$(3)$ If $\omega^\epsilon(t)$ is the (local) solution of the 3D Euler equations with $\omega^{\epsilon}-\bf{e}_3$ in $L^{\infty}([0,C_{2}];B^{\frac{3}{4}}_{4,1})$ with $\omega^\epsilon (0) = \omega_{0}^\epsilon,$
then there exists some $t\in (0,\epsilon]$ so that
$$|\omega(t)^\epsilon-{\bf{e_3}}|_{L^\infty} \geq C_{3}.$$

\end{thm}

\begin{rema} The proof of Theorem 6.1 follows the same ideas we used in the proof of Theorem 5.1 except that the critical space in two dimensions is $B^{\frac{1}{2}}_{4,1}$ while the critical space in three dimensions is $B^{\frac{3}{4}}_{4,1}$. \end{rema}

\begin{rema} One might be concerned by the fact that $\bf{e_3}$ is not of finite energy in the whole space; however, the result is very easily localized by considering ${\bf e_3}$ multiplied by a smooth cut-off function. 
\end{rema}

\section{The Euler equations with $C^{1}$ data}  

As another bi-product of Proposition 3.1, the incompressible Euler equations are \emph{strongly} ill-posed for $u\in C^1 \cap L^2.$ Indeed, consider the incompressible Euler equations in velocity form:
\begin{equation} \partial_t u + (u\cdot \nabla) u +\nabla p =0, \end{equation}
\begin{equation} \text{div}(u)=0. \end{equation}
Notice that the equation for the gradient of $u$ is:
\begin{equation}\partial_t\nabla u + (u\cdot \nabla) \nabla u +D^2 p+ \nabla u:\nabla u =0. \end{equation}
The pressure is recovered from $u$ by the following equation:
\begin{equation} \Delta p =\text{div} ((u\cdot \nabla) u)= \sum_{l\not=k} u_{l,k}u_{k,l}, \end{equation} with $u_{j,i} =\partial_{x_i}u_{j}.$
Then notice that $D^2 p= R_{i}R_{j}(\sum_{l\not=k} u_{l,k}u_{k,l}).$ Therefore, (7.3) becomes:
$$\partial_t\nabla u + (u\cdot \nabla) \nabla u +R_{i}R_{j}(\sum_{l\not=k} u_{l,k}u_{k,l})+ \nabla u:\nabla u =0. $$
We will write this as:

\begin{equation} \partial_t\nabla u + (u\cdot \nabla) \nabla u +R(B(\nabla u, \nabla u))+ Q(\nabla u, \nabla u) =0, \end{equation}

where $R:=(R_{i}R_{j})_{i,j}$ is a matrix of singular integral operators, $B(\nabla u, \nabla u):= \sum_{l\not=k} u_{l,k}u_{k,l},$ and $Q(\nabla u,\nabla u)=\nabla u:\nabla u.$

We have the following theorem:

\begin{thm} For every $\epsilon>0,\delta>0$ small enough there exists $u_{0} \in C^\infty(\mathbb{R}^2),$ of compact support, with $$|u_{0}|_{C^1\cap L^2} \leq \epsilon,$$ such that if we denote by $u(t)$, the solution of the incompressible Euler equations in $\mathbb{R}^2$ with initial data $u_0,$ then
$$\sup_{0<t<\delta} |u(t)|_{C^1\cap L^2} \geq \frac{1}{\epsilon}.$$
\end{thm}

In section 8 we will prove a stronger result: 

\begin{thm}
For every $\epsilon>0,\delta>0$ small enough there exists $u_{0} \in C^1\cap L^2(\mathbb{R}^2),$ with $$|u_{0}|_{C^1\cap L^2} \leq \epsilon,$$ such that if we denote by $u(t)$ the unique (Yudovich) solution of the incompressible Euler equations in $\mathbb{R}^2$ with initial data $u_0,$ then
$$\sup_{0<t<\delta} |u(t)|_{C^1\cap L^2}=+\infty.$$
\end{thm}

\begin{rema}The growth in the $C^1$ case will come from the singular integral which arises in the pressure term.  However, we will have to be careful because the pressure term is not linear in $u,$ but bilinear. \end{rema}
\begin{rema}The construction in Theorem 7.1 is completely local. Therefore, the result holds on a bounded domain as well as on the torus. \end{rema}
\begin{rema}With the exception of choosing the right initial data, the proof of theorem 7.1 is quite soft--so it likely can be used in several other contexts. \end{rema}
\begin{rema}After the completion of this work we came to know that Misiolek and Yoneda \cite{MisY} have proven ill-posedness for the Euler equations in $C^1$ in the sense that the solution map could not be continuous. Their result is not about norm inflation but about discontinuity of the solution map. Their method relies upon a clever adaptation of the work of Bourgain and Li \cite{BLi}; it does not seem that there is any apparent relation between our work and theirs.\end{rema} 

\addtocontents{toc}{\SkipTocEntry}

\subsection{A toy model}\label{ToyModel}

To understand the effect of the pressure term, $R(B(\nabla u, \nabla u)),$ we may consider the following toy model: $$\partial_t f = R(f^2).$$ We want to see that this model is ill-posed on $L^\infty.$  In the case of $f_t = R(f)$ we are able to solve this equation on the Fourier-side by a series expansion in order to deduce that $$|f|_{L^\infty} \geq |f_0 +t R(f_0)|_{L^\infty} -{t^2} C|f_0|_{B^{\frac{1}{2}}_{4,1}}.$$ However, in the case where we have $f_t = R(f^2),$ it is not clear how to solve the equation using any sort of similar expansion. 

\begin{prop}
Let $B$ be a quadratic form acting on matrices. Consider the following matrix PDE:
\begin{equation} \label{QRTP1}f_{t}=R(B(f,f)),\end{equation}
\begin{equation} \label{QRTP2}f(0,x)=f_0 (x),\end{equation}  
where $R$ is a Calder\'on-Zygmund singular integral operator. Then, \eqref{QRTP1}-\eqref{QRTP2} is locally well-posed on $B^{a}_{\rho,1}$ for all $a\rho\geq n.$ Moreover, for $t$ small, smooth solutions satisfy the following bounds:
\begin{equation} |f(t)|_{L^\infty} \geq |f_0 +t R(B(f_0,f_0))|_{L^\infty} - t^2 C( \sup_{0\leq \tau\leq t}|f(\tau)|_{L^\infty}) |f_0|_{B^{\frac{1}{2}}_{2d,1}}. \end{equation}
\end{prop} 

\begin{rema} Bound (7.8) only holds so long as $f$ exists. However, note that if the initial data $f_0$ belongs to $B^{a}_{\rho,1}$ with $a\rho\geq 1,$ then finite-time blow up in  \eqref{QRTP1}-\eqref{QRTP2} can \emph{only} happen if $|f|_{L^\infty}$ blows up. This will be important in what follows. \end{rema}

\begin{proof}

The local well-posedness is standard. Indeed, all that is needed is that $R$ is bounded on $B^{a}_{\rho,1}$ and that these spaces are algebras containing $L^\infty.$ Indeed, recall the following inequality:
$$|fg|_{B^a _{\rho,1}} \leq |f|_{B^a _{\rho,1}}|g|_{L^\infty}+|f|_{L^\infty}|g|_{B^a _{\rho,1}}.$$
First we write: 
\begin{equation} f_{t}=R(B(f_0,f_0))+\big ( R(B(f,f))-R(B(f_0,f_0)) \big ),\end{equation}
Next, note that $|f_t|_{L^\infty_{t,x}} \leq C|B(f,f)|_{L^\infty_{t}B^{\frac{1}{2}}_{2d,1}} \leq C(|f|_{L^\infty_{t,x}})|f_0|_{B^{\frac{1}{2}}_{2d,1}},$ by local well-posedness. 
Consequently, 
$$ |B(f,f)-B(f_0,f_0)|_{{B^{\frac{1}{2}}_{2d,1}}} \leq C(|f|_{L^\infty_{t,x}}) |f-f_0|_{{B^{\frac{1}{2}}_{2d,1}}} \leq tC(|f|_{L^\infty_{t,x}})|f_0|_{{B^{\frac{1}{2}}_{2d,1}}}.  $$ 
Hence, so long as the solution $f(t)$ exists, 
\begin{equation} |f(t)|_{L^\infty} \geq |f_0 +t R(B(f_0,f_0))|_{L^\infty} - t^2 C( \sup_{0\leq \tau\leq t}|f(\tau)|_{L^\infty}) |f_0|_{B^{\frac{1}{2}}_{2d,1}}. \end{equation}
\end{proof}

\begin{cor}
Let $B$ be a quadratic form acting on matrices. Consider the following matrix PDE:
\begin{equation} \label{QRTP3}f_{t}=R(B(f,f))+g,\end{equation}
\begin{equation} \label{QRTP4}f(0,x)=f_0 (x)\end{equation}  
where $R$ is a Calder\'on-Zygmund singular integral operator and $g$ is a given function belonging to $L^\infty_{t}B^{a}_{\rho,1},$ with $a\rho\geq 1$. Then, \eqref{QRTP3}-\eqref{QRTP4} is locally well-posed on $B^{a}_{\rho,1}$ for all $a\rho\geq n.$ Moreover, for $t$ small, smooth solutions satisfy the following bounds:
\begin{equation} |f(t)|_{L^\infty} \geq |f_0 +t R(B(f_0,f_0))|_{L^\infty} -t|g|_{L^\infty}- t^2 C( \sup_{0\leq \tau\leq t}|f(\tau)|_{L^\infty}) |f_0|_{B^{\frac{1}{2}}_{2d,1}} \end{equation}
\end{cor}

We are now in a position to prove Theorem 7.1. 
\\

\emph{Proof of Theorem 7.1.}
\\

Call $f:=\nabla u$ and recall that $\text{div}(u)=0.$ Towards a contradiction, suppose that for all $f$ with $|f_0|<\epsilon,$ $\sup_{0<t<\delta}|f(t)|_{L^\infty}\leq M,$ for some given $\epsilon, \delta, M.$  Note if the assertion is true, we can solve the $n$-dimensional Euler equations on $[0,\delta]$ for any initial data with $|\nabla u_0|=|f_0|_{L^\infty}<\epsilon.$ 
Then, $f$ satisfies the equation:
$$f_t +(u\cdot\nabla) f + Q(f,f) +R(B(f,f))=0.$$
Now we write this equation along the characteristics of $u$ by solving $$\dot{\Phi}=u(\Phi)$$ $$\Phi(0)=Id.$$
Then we get: 
$$(f\circ \Phi)_t+ Q(f\circ\Phi, f\circ\Phi)+R(B(f\circ\Phi,f\circ\Phi))+[R,\Phi]B(f,f)=0.$$

By Corollary 7.4, we have:
$$|f\circ \Phi|_{L^\infty} \geq |f_0 +t R(B(f_0,f_0))|_{L^\infty} -t|g|_{L^\infty}- t^2 C( \sup_{0\leq \tau\leq t}|f(\tau)|_{L^\infty}) |f_0|_{B^{\frac{1}{2}}_{2d,1}}, $$ where $$g:=Q(f\circ\Phi, f\circ\Phi)+[R,\Phi]B(f,f).$$ Here, we have implicitly used the result of Vishik \cite{V1} that the Euler equations are locally well-posed on $B^{\frac{1}{2}}_{2d,1}$ which implies that the remainder term, $g$, belongs to $B^{\frac{1}{2}}_{2d,1}.$ 
Now we need to estimate $g$ using the commutator estimate (3.1).  Since, $|f(t)|=|\nabla u(t)|\leq M$ on $[0,\delta],$ we can choose $t$ very small so that the conditions of Proposition 3.1 are satisfied (namely, that $\Phi$ is sufficiently close to the identity). Hence, we have that
$$|g|_{L^\infty} \leq C|f|_{L^\infty}^2+tC(|\nabla u|_{L^\infty})|B(f,f)|_{B^{\frac{1}{2}}_{2d,1}} \leq C|f|_{L^\infty}^2 +tC(|\nabla u|_{L^\infty})|f_0|_{{B^{\frac{1}{2}}_{2d,1}}} |f|_{L^\infty}.  $$
Consequently, we have:
$$|f|_{L^\infty} \geq |f_0 +tR(B(f_0,f_0))|_{L^\infty}- t \big( C|f|_{L^\infty}^2 +tC ( |\nabla u|_{L^\infty})|f_0|_{{B^{\frac{1}{2}}_{2d,1}}}\big).$$

By assumption, $|f|_{L^\infty}< M.$
Hence, 
$$|f|_{L^\infty} \geq |f_0 + tR(B(f_0,f_0))|_{L^\infty}-tC(M)-t^2C(M)|f_0|_{B^{\frac{1}{2}}_{2d,1}}.$$

\begin{lemm}
There exists a sequence of divergence-free functions $g_N\in C^\infty $, of compact support, such that the following holds:
\begin{equation} | \nabla g_{N}|_{L^\infty} \leq 1,\end{equation}
\begin{equation} |R(B(\nabla g_{N},\nabla g_N))|_{L^\infty} \geq cN,\end{equation}
\begin{equation} |g_{N}|_{B^{\frac{1}{2}}_{2d,1}} \leq CN, \end{equation}
where $c$ and $C$ are constants independent of $N$. 
\end{lemm}

Assuming this lemma is true, take $u_0 = \epsilon g_N,$ where $N$ is fixed for the moment. Then, 
$$|f|_{L^\infty} \geq ctN\epsilon^2 -\epsilon-tC(M)-t^2 C(M) N. $$
Recall that we need $t<\frac{c}{M} $ in order to apply Proposition 3.1 (because we need $\Phi$ to be close enough to the identity). Now choose $N$ large enough, $t$ small enough and then $|f|_{L^\infty} >M,$ which is a contradiction. Consequently, for every $\epsilon,\delta,M>0,$ there exists $u_0 \in C^\infty$ with $|u_0|_{\text{Lip}} \leq \epsilon$ and $$\sup_{0\leq t\leq \delta}|\nabla u(t)|\geq M.$$
\qed

\addtocontents{toc}{\SkipTocEntry}

\subsection{Proof of Lemma 7.5}

We are interested in showing that for some $i,j$ and for some divergence free $u,$ with $\nabla u \in L^\infty,$ $D^2 p= R_{i}R_{j} \text{det}(\nabla u)$ has a logarithmic singularity. Once that is shown, Lemma 7.5 will follow by a regularization argument.  
Take a harmonic polynomial, $Q,$ which is homogeneous of degree 4. In the two-dimensional case, we can take $$Q(x,y):=x^4+y^4-6x^2y^2,$$
$$\Delta Q =0.$$

Define $$G(x,y):= Q(x,y) Log (x^2 +y^2).$$
Notice, on the one hand, we have \begin{equation} \partial_{i}\partial_{j}\Delta G\in L^\infty (B_{1}(0)), i,j\in\{1,2\}.\end{equation}
On the other hand, we have \begin{equation} \partial_{xxyy} G=-24Log(x^2+y^2)+H(x,y),\end{equation} with $H\in L^\infty(B_{1}(0)).$ In particular, $\partial_{xxyy}G$ has a logarithmic singularity at the origin--and the same can be said about $\partial_{xxxx}G$ and $\partial_{yyyy}G.$ 

\vspace{3mm}

Define $\tilde{u}=\nabla^\perp \Delta G.$ Then, by (7.17),  $\nabla \tilde{u} \in L^\infty(B_{1}(0)).$ Moreover, by definition, $$R_{i}R_{j} \nabla \tilde{u}=\nabla\nabla^\perp\partial_{ij}G.$$
Thus, for example, $R_{1}R_2 \partial_x\tilde{u}_{1}=\partial_{xxyy}G$ has a logarithmic singularity in $B_{1}(0).$ Unfortunately, we are interested in showing that $R_{i}R_{j} \text{det}(\nabla u)$ has a logarithmic singularity for some $i,j,$ not $R_{i}R_{j}\nabla u.$ To rectify this, we choose $$u=\delta \nabla^\perp\Delta (\chi G)+\eta\nabla^\perp(y\chi),$$ where $\eta, \delta$ are small parameters which will be determined and $\chi$ is a smooth cut-off function with: 
$$\chi=1 \, \, \text{on} \, \, B_1(0),$$
$$\chi=0 \,\, \text{on} \, \, B_2(0)^c,$$ and $$|\nabla^2\chi|_{L^\infty}\leq 2.$$ 
Note that $u$ is divergence free and $$u= \delta\nabla^\perp \Delta G +  \eta(y,0) \,\, \text{on} \,\, B_{1}(0). $$
Therefore, on $B_1(0)$, $$\nabla u= \delta \left[ {\begin{array}{cc} -\partial_{xy} \Delta G & -\partial_{yy}\Delta G \\ \partial_{xx}\Delta G & \partial_{xy}\Delta G \\\end{array} } \right] + \eta   \left[ {\begin{array}{cc}0 & 1 \\0 & 0 \\\end{array} } \right] .$$
Hence, $$\text{det}(\nabla u) = \eta\delta\partial_{xx}\Delta G + \delta^2 J(x,y),$$ where $J$ is bounded on $B_{1}(0)$.
Now consider $R_2 R_2 \text{det}(\nabla u):$
$$R_2 R_2 \text{det}(\nabla u)= \eta \delta \partial_{xxyy}G +\delta^2 R_{2} R_2 J.$$
Now, by (7.18), we have $$ R_2 R_2 \text{det}(\nabla u)= \eta \delta (-24Log(x^2+y^2)+H(x,y)) +\delta^2 R_{2} R_2 J, $$
with $H$ and $J$ bounded. Recall that $R_2R_2$ maps $L^\infty$ to BMO and that any BMO function can have at most a logarithmic singularity\footnote{Using the John-Nirenberg inequality, any $L^1\cap$ BMO function, $f$, satisfies $$|f|_{L^p}\leq C\, p |f|_{L^1\cap BMO}.$$ Moreover, $|\log(x^2+y^2)|_{L^p}\approx c\,p$ for some fixed constants $c$ and $C$.}.

Thus,  
$$|R_2 R_2 \text{det}(\nabla u)|\geq 24 \eta \delta |Log(x^2+y^2)| -C\delta^2 |Log(x^2+y^2)| -|H(x,y)|.$$
Now we may choose $\delta < < \eta$  so that $$|R_2 R_2 \text{det}(\nabla u)|\geq \alpha Log(x^2+y^2),$$ for some fixed number $\alpha,$ while $|\nabla u|_{L^\infty}\leq 1$.

One may regularize the constructed velocity field by replacing $Log(x^2+y^2)$ with $Log(x^2+y^2+\frac{1}{2^N})$ or by convolving $u$ with an approximation of the identity.   

\section{Strong ill-posedness in $C^1:$ the $L^p$ approach}

It is possible to prove the ill-posedness of the Euler equations in $C^1$ in a more direct fashion. We now prove Theorem 7.2.

\begin{proof}

Using the initial data constructed above in section 7, we see that there exists $u_0$ so that 
$$u_0 \in C^1$$ but $$|D^2 p_0|_{L^p} =|B(\nabla u_0, \nabla u_0)|_{L^p}\geq c p^{3/4}, \forall p>1.$$
Furthermore, as was noted in Proposition 3.3, $$\|[R,\Phi]\|_{L^p\rightarrow L^p} \leq c_{p} |\Phi-I|_{Lip}$$ and
 $ c_{p}\approx p$ for $p$ large. 
Therefore, if we assume that the solution $u(t)$ remains Lipschitz for positive time (say that $|\nabla u|\leq M$ for $t<c$) then we see that $\nabla u$ will satisfy the following estimate in $L^p$.
$$|\nabla u|_{L^p} \geq c p^{3/4}t-C(M) p t^2$$ for all $p$ and all $t>0.$
This obviously leads to a contradiction for small $t$ since $|\nabla u(t)|_{L^p}>c p^\frac{1}{4}$ for $t<\frac{1}{\sqrt{p}}$ for some small $c$ while $\nabla u $ remains bounded. Thus the solution must leave $C^1$.
Note that our initial data can be taken to be as localized as we want so we can deal with the whole space, periodic boundary conditions, and the bounded domain case. 

\end{proof}

\section{The $C^k$ case}

\begin{thm}
The Euler equations are strongly ill-posed in $C^k$ spaces for $k\geq 1$. In other words, for every $\epsilon>0$ there exists initial data $u_0\in C^k$ such that the unique solution, $u(t)$, of the Euler equations with initial data $u_0$ leaves $C^k$ immediately. 

\end{thm}

We note that very recently Bourgain and Li have proven the same result as above \cite{BLi2}. We clarify here that strong ill-posedness in $C^k$ can be proven quite easily only using commutator estimates without having to rely upon very intricate constructions. 

\begin{proof}
We just sketch the proof since it is basically the same as the $C^1$ case. 
Note that it suffices to consider the two dimensional Euler equations (in the whole space case in higher dimensions a similar argument can be made simply by modifying the initial data slightly).  
Now consider the equation for $D^k u:= \partial_{x}^k u$ which means $k$ spatial derivatives of $u$ with respect to the first variable. 

$\nabla D^{k-1} u$ satisfies the following equation:

$$\partial_t \nabla D^{k-1} u + u \cdot \nabla D^k u +\sum_{j,l}^k Q(D^j u, D^l u) +D^{k-1} D^2 p=0. $$

We are going to take data in $C^{k}.$ Then, locally in time, there will be a $C^{k-\epsilon}$ solution by the result of Lichtenstein \cite{Lich}. Assume that this solution remains in $C^k$ for $t\in [0,1].$  
Now recall that $$\Delta p= det (\nabla u)$$ so that $$(D^2 p)_{ij} =\big (R_{i}R_{j} det(\nabla u) \big )_{ij}.$$
Following the proof of Theorem 8.1, it suffices to construct $u_0\in C^k$ such that $|D^{k+1} p_0|_{L^p}\geq cp$ as $p\rightarrow\infty.$ Notice that $D^{k+1}p_0$ will consist of many terms all of which belong to $C^{2-\epsilon}$ except for the terms where all of the derivatives hit one column of $\nabla u$ so that we only have to focus on these terms (because the $C^{2-\epsilon}$ terms will be well-controlled) Now we can choose $P$ to be the $k+3$ degree homogeneous polynomial which is just the $k^{th}$ integral with respect to $x$ of the $Q$ constructed in section 7. Then the argument is the same as in section 7 and we are done.

\end{proof}

\section{Further Results}

In this section, we collect some further applications of Theorem \ref{LowerBoundTheorem} and Proposition \ref{CommutatorEstimate}. Each application has a slightly different complication which we must overcome first to apply the commutator estimate. First, each of the examples we give are systems and not scalar equations. Second, in the Oldroyd-B case, we will see that the "singular integral" on the right hand side of $$\partial_t f+u\cdot\nabla f=R(f)$$ may actually be non-linear so long as the non-linear part can be controlled in the right way using a-priori estimates. In the SQG case, we will also find that the singular integral $R$ can actually have a variable-coefficient and that this can be overcome by using local well-posedness in a critical Besov space coupled to a Taylor expansion in time after one factors out the effect of the transport term (which is how Proposition \ref{CommutatorEstimate} is used).

\subsection{Oldroyd-B}

Recall the two-dimensional Oldroyd-B system which models the evolution of the velocity field, $u$, and strain matrix, $\tau,$ of some non-Newtonian fluids:

\begin{equation}
\label{OldroydB1}\partial_t u+u\cdot \nabla u+\nabla p =\Delta u + \div(\tau)
\end{equation}
\begin{equation}
\label{OldroydB2}\div(u)=0
\end{equation}
\begin{equation} \label{OldroydB3}\partial_t\tau+u\cdot\nabla\tau +Q(\nabla u,\tau)+a\tau=D(u)\end{equation} with $Q(\nabla u,\tau)= \tau(\nabla u-\nabla u^t)$,  $D(u)=\frac{1}{2}(\nabla u+\nabla u^t)$, and $a\geq 0$. It is known since the work of Chemin and the second author \cite{CM} that to prove global regularity for this system one needs $L^\infty$ bounds on $\tau$ in the sense that a bound on $\int_0^T |\tau(s)|_{L^\infty}ds$ actually implies that smooth solutions on a time interval $[0,T)$ can be continued past $T.$ A natural question one could ask is whether it is possible to prove local well-posedness for merely bounded strain matrix $\tau$. Using Theorem \ref{LowerBoundTheorem}, we will in fact show that even if $u_0\equiv 0$, $\tau$ can start arbitrarily small in $L^\infty$ and become of size 1 in arbitrarily short time. That is, this system will be shown to be mildly ill-posed. 

\begin{thm}\label{OldroydBThm}
There exists a universal constant $c>0$ and a sequence of initial strain-matrices $\tau_0^\epsilon\in C^\infty(\mathbb{R}^2)$ with $|\tau_0^\epsilon|_{L^1\cap L^\infty}<\epsilon$ but which satisfy that the unique local solution $\tau^\epsilon$ to \eqref{OldroydB1}-\eqref{OldroydB3} with $u_0^\epsilon\equiv 0$ satisfies:
$$|\tau^\epsilon(t)|_{L^\infty}>c$$ for some $t<\epsilon$. 
\end{thm}

The proof will be similar to the proof of Theorem \ref{RTP1}, though we will need to be more careful regarding certain issues. First, we must consider the linearized\footnote{We have actually added the linear term $\tau$ to the second equation to simplify this sketch.} system:
$$\partial_t u+\nabla p= \Delta u +\div(\tau),$$
$$\partial_t \tau =D(u)+\tau.$$
By inspecting the linearized equation, we find that the quantity $\Gamma:=\omega+\Delta^{-1}\div\curl(\tau)$, first introduced in \cite{ER},  actually satisfies the heat equation. Hence, as long as the non-linear terms can be seen to be sub-critical in front of the Laplacian, $\Gamma$ can actually be shown to be smoother than expected. In fact, we will be able to show that $\Gamma$ is $C^\frac{1}{2}$ regular for $t>0$ even if the initial data is only bounded. This means that, up to a smoother term, $\omega\approx -\Delta^{-1}\div\curl(\tau).$ 

We then get that \eqref{OldroydB3} can be written as:
$$\partial_\tau+u\cdot\nabla\tau +Q(R_1(\tau),\tau)=R_2(\tau)+G,$$ where $G$ are "good" terms and where $u$ can be determined from $\tau$ by a pseudo-differential operator of order $-1$ plus a smoother correction. From that point the proof will follow closely the proof of Theorem \ref{RTP1}. Let us also remark that we are implicitly assuming that the Oldroyd-B equation is locally well-posed for $\tau$ in the space $B^{\frac{1}{2}}_{4,1},$ but in fact this is a consequence of showing that $\Gamma$ is smoother than expected and following the same proof as in Lemma \ref{LWP} and the work of Vishik on the 2d and 3d Euler equations \cite{V1}.  

\begin{proof}[Proof of Theorem \ref{OldroydBThm}]

\

\vspace{3mm}

\emph{Step 1: Estimates for the Good Quantity}

\vspace{3mm}

\noindent First we pass to the vorticity formulation of \eqref{OldorydB1}: 
$$\partial_t\omega+u\cdot\nabla\omega = \Delta\omega+\curl \div(\tau).$$ Next, define the operator $R_0:=\Delta^{-1} \div\curl$ acting on matrices and apply it to \eqref{OldroydB3}. We then get:
$$\partial_t R_0\tau+R_0(u\cdot\nabla \tau)+R_0Q(\nabla u, \tau)=-\omega.$$
Now define $\Gamma:=\omega+R_0\tau$ and notice:
$$\partial_t \Gamma = \Delta \Gamma,$$ with $$N(u,\tau):=\omega+a\tau+u\cdot\nabla\omega +R_0(u\cdot\nabla \tau)+R_0(Q(\nabla u,\tau)).$$ 
Next, using Duhamel's formula we see:
\begin{equation}\label{Duhamel} \Gamma= e^{t\Delta} \Gamma_0-\int_0^t e^{(t-s)\Delta}N(u,\tau)(s)ds.\end{equation}
Now, as is established in \cite{LMOldroyd}, $|\omega|_{L^8}+|\tau|_{L^8}\leq (|\tau_0|_{L^8}+|\omega_0|_{L^8})\exp(Ct)$. This, in turn, implies that $|\omega|_{L^8}+|\tau|_{L^8}\leq C\epsilon$  since we will choose $|\tau_0|_{L^1\cap L^\infty}<\epsilon$ and $\omega_0\equiv 0$, and $t<1$.  Now note that using these estimates $$|N(u,\tau)|_{W^{-1,4}}<C\epsilon.$$ In particular, using standard parabolic estimates, and the Sobolev imbedding theorem
$$|\Gamma(t)|_{C^{\frac{1}{2}}}\leq C\frac{|\Gamma_0|_{L^\infty}}{\sqrt{t}}.$$ Note that the degeneration of the bound as $t\rightarrow 0$ is only coming from the linear part. 

\vspace{3mm}

\emph{Step 2: Application of Theorem \ref{LowerBoundTheorem}}

\vspace{3mm}

We can write \eqref{OldorydB3} as:
$$\partial_t\tau +u\cdot\nabla\tau+Q(\nabla u,\tau)+a\tau = D(\nabla^\perp(-\Delta)^{-1} \omega).$$ Now we introduce $\Gamma$:
$$\partial_t\tau +u\cdot\nabla\tau+Q(\nabla u,\tau)+a\tau = -D(\nabla^\perp(-\Delta)^{-2}\div\curl(\tau))+D(\nabla^\perp(-\Delta)^{-1} \Gamma).$$
Now define $R:=-D(\nabla^\perp (-\Delta)^{-2}\div\curl)$ and we see:
$$\partial_t\tau +u\cdot\nabla\tau = R(\tau)+G$$ with the good part defined by:
$$G:=D(\nabla^\perp(-\Delta)^{-1} \Gamma)-Q(\nabla u,\tau)-a\tau.$$
Just as before, we get the following lower bound on the growth of $\tau$ using a slight modification of Theorem \ref{LowerBoundTheorem} to include the good term $G$ 
$$|\tau|_{L^\infty}(t)\geq |t R(\tau_0)+\tau_0|_{L^\infty}-Ct^2(1+|u|_{L^\infty Lip} \exp(tC|u|_{L^\infty Lip}))|\tau_0|_{B^{\frac{1}{2}}_{4,1}}-|G(t)|_{B^{\frac{1}{2}}_{4,1}}.$$

To conclude the proof of the theorem, it suffices to show that $R$ satisfies Assumption \ref{Assumption1} as in the proof of Theorem \ref{RTP1}. 

\vspace{3mm}

\emph{Step 3: Verifying that $R$ satisfies Assumption \ref{Assumption1} for properly chosen initial data.} 
\
\vspace{3mm}

As above, $R(\tau_0)=-D(\nabla^\perp(-\Delta)^{-2}\div \curl)(\tau_0)$ and $\tau_0$ is a matrix of functions. Take the case where $$\tau_0= \left[ {\begin{array}{cc} a_0 & 0 \\ 0 & 0 \\\end{array} } \right]$$ for some smooth function $a_0$. Then, $R(\tau_0)=-D(\nabla^\perp(-\Delta)^{-2}\partial_{xy}a_0).$ This means that the components of $R(\tau_0)$ are just fourth order Reisz transforms of $a_0$ like $(R_1^2-R_2^2)R_1R_2 a_0.$ Since Riesz transforms are bounded on the spaces $B^{\frac{1}{2}}_{4,1},$ it suffices to show that $(R_1^2-R_2^2)R_1R_2$ satisfies Assumption \ref{Assumption1}. This was already done above in the study of the Euler equation for $C^1$ velocity fields and we recall that the right initial data to choose is:

$$a_0^\epsilon(x,y)=\epsilon\Delta^2 (xy(x^2-y^2) \log(x^2+y^2+\frac{1}{N^2})\phi(x^2+y^2))$$ with $\phi$ a smooth function with $\phi\equiv 1$  in a neighborhood of 0 and $\phi\equiv 0$ outside of $B_2(0)$.  Since $xy(x^2-y^2)$ is harmonic, it is easy to see that $|a_0|_{L^\infty}\approx \epsilon$ but that $|(R_1^2-R_2^2)R_1R_2 a_0|_{L^\infty}\approx|a_0|_{B^{\frac{1}{2}}_{4,1}}\approx \epsilon\log N.$

\vspace{3mm}

\emph{Step 4: Choosing the constants}

\vspace{3mm}

\noindent Combining the previous steps leads us to the following lower bound:
$$|\tau|_{L^\infty}(t)\geq tc\epsilon \log(N) -Ct^2\epsilon^2 \log(N)^2 -\frac{\epsilon}{\sqrt{t}}.$$ 
First choose $\epsilon t\log(N)=\frac{c}{2C}.$
Then, $$|\tau|_{L^\infty}\geq \frac{c^2}{4C}-\frac{\epsilon}{\sqrt{t}}.$$
Now we take $t=\epsilon$ and we are done. 
\end{proof}

\subsection{The SQG Equation}

In this section we will prove that basic stationary solutions to the SQG are unstable in $L^\infty$ in the same sense as what we did for the 3d Euler equation in Section \ref{3dEuler}. Recall the surface quasi-geostrophic equation on $\mathbb{R}^2$:
\begin{equation}\label{SQG1} \partial_t \theta+ u\cdot\nabla \theta=0, \end{equation}
\begin{equation}\label{SQG2} u=\nabla^\perp (-\Delta)^{-\frac{1}{2}}\theta,\end{equation} for the active scalar $\theta:\mathbb{R}^2\times\mathbb{R}\rightarrow \mathbb{R}$.
This system originally appeared as a model in atmospheric science (\cite{CMT},\cite{M}) but is also seen as a good model for the 3d Euler equation since the quantity $\nabla^\perp \theta$ obeys a system very similar to the 3d vorticity equation. Like the 3d Euler equation, the global regularity problem is still outstanding though exciting advances have been made in recent years in the direction of singularity formation (\cite{KRYZ}). 

Recall that, just like the 2d Euler equation, $\theta(x,t)=G(x_2)$ is a stationary solution to \eqref{SQG1}-\eqref{SQG2} on $\mathbb{T}^2$  for any smooth mean-zero function $G\not\equiv 0$. We will prove that \emph{for any} $C^{2,\alpha}$ smooth $G$, there exists a sequence of data $\theta_0^\epsilon\rightarrow G$ in $W^{1,\infty}$ but such that $|\theta^\epsilon(t)-G|_{W^{1,\infty}}>c$ for some $t<\epsilon.$ In fact, the same can be done for any smooth stationary solution which is not identically constant. However, we do not pursue this here. To do this, we write the evolution of a perturbation of such a stationary solution $G:$
\begin{equation}\label{SQGPerturbed} \partial_t \theta+ u\cdot\nabla\theta+H(G)\partial_{x_1}\theta=u_2 G',\end{equation} where $H=\frac{d}{dx_2}(-\frac{d^2}{dx_2^2})^{-1/2}$ is the Hilbert transform in the $x_2$ variable. Equations \eqref{SQGPerturbed}-\eqref{SQG2} control the evolution of perturbations of the stationary solution $G$. Using a variant on Theorem \ref{LowerBoundTheorem}, we prove:

\begin{thm}\label{SQGThm}
There exists a constant $c>0$ proportional to the infimum of $|G'|$ and a sequence of mean-zero initial data $\theta_0^\epsilon\in C^\infty(\mathbb{T}^2)$ with $|\theta_0^\epsilon|_{W^{1,\infty}}<\epsilon$ but such that the associated unique local solution $\theta^\epsilon$ to \eqref{SQGPerturbed}-\eqref{SQG2} satisfies:
$$|\theta^\epsilon(t)|_{W^{1,\infty}}>c$$ for some $t<\epsilon$. 
\end{thm}

The proof is similar to those done before except, like the Oldroyd-B case, we must take some care to understand the linearized system first. 

\begin{proof} The proof will proceed in several steps similar to Theorem \ref{OldroydBThm}

\

\vspace{3mm}

\emph{Step 1: Lower bound in the linear case.}
We begin by looking at the linearized system: $$\partial_t\theta + H(G)\partial_{x_1} \theta=G'u_2,$$
$$u=\nabla^\perp(-\Delta)^{-\frac{1}{2}}\theta.$$
Using the relation between $u$ and $\theta,$ we see:
$$\partial_t\theta + H(G)\partial_{x_1} \theta=G'R_1(\theta),$$ where $R_1$ is the first Riesz transform. First consider 
$$\partial_t\theta = G'R_1(\theta).$$  As we did in the analysis of the toy model in Section \ref{ToyModel}, we wish to show that 
$$|\theta- \theta_0-tG'R_1(\theta_0)|_{B^{\frac{1}{2}}_{4,1}}\leq C t^2|\theta_0|_{B^{\frac{1}{2}}_{4,1}}.$$ It is easy to show that this is globally well-posed on $B^{\frac{1}{2}}_{4,1}$ since $G\in C^{1,\alpha}$ and that $|\theta|_{B^{\frac{1}{2}}_{4,1}}\leq |\theta_0|_{B^{\frac{1}{2}}_{4,1}}\exp(Ct).$ Then we observe that $$|\theta|_{B^{\frac{1}{2}}_{4,1}}+|\partial_t\theta|_{B^{\frac{1}{2}}_{4,1}}\leq C |\theta_0|_{B^{\frac{1}{2}}_{4,1}}$$ on the interval $t\in [0,1],$ where $C$ is a constant that depends on $G$. Now notice: $$\theta-\theta_0=\int_0^t G'R_1(\theta)ds=\int_0^t G'(R_1(\theta)(s) -R_1(\theta_0))ds+tG'R_1(\theta_0),$$ which implies $$|\theta- \theta_0-tG'R_1(\theta_0)|_{B^{\frac{1}{2}}_{4,1}}\leq C t^2|\theta_0|_{B^{\frac{1}{2}}_{4,1}}.$$ This implies the linear lower bound:
$$|\theta|_{L^\infty}> t|G'R_1(\theta_0)|_{L^\infty}-|\theta_0|_{L^\infty}-Ct^2|\theta_0|_{B^{\frac{1}{2}}_{4,1}}.$$

\vspace{3mm}

\emph{Step 2: Lower bound in the nonlinear case.}

\vspace{3mm}

Now let's return to the original system:
$$\partial_t\theta + \tilde u \cdot\nabla \theta=G'R_1\theta$$
where $\tilde u=u+(H(G),0)$.
First we differentiate the system and set $F=\nabla^\perp\theta$:
$$\partial_tF+ \tilde u\cdot\nabla F=G'R_1F+ F\cdot\nabla \tilde u+\nabla^\perp G'R_1\theta.$$
Next we write the equation along the flow of $\tilde u$. Let $\Phi$ be the Lagrangian flow-map associated to $\tilde u$. Then, $$\partial_t(F \circ\Phi)=(G'\circ\Phi)R_1(F\circ\Phi) + (G'\circ\Phi)[R_1,\Phi]F +(F\cdot\nabla \tilde u+\nabla^\perp G'R_1\theta)\circ\Phi.$$
Taking the $B^{\frac{1}{2}}_{4,1}$ norm of both sides of this equality and using local well-posedness for $F$ in $B^{\frac{1}{2}}_{4,1}$, we get:
$$|\partial_t(F\circ\Phi)|_{B^{\frac{1}{2}}_{4,1}}\leq C |F_0|_{B^{\frac{1}{2}}_{4,1}}+|F_0|_{B^{\frac{1}{2}}_{4,1}}^2,$$
 so long as $t<\frac{c}{|F_0|_{B^{\frac{1}{2}}_{4,1}}}$. 
Now upon integrating both sides of the equation for $F\circ\Phi$, arguing as in the linear case above, and using Proposition \ref{CommutatorEstimate} we get:
$$|F\circ\Phi-F_0-t(G' R_1(F_0)+F_0\cdot\nabla\tilde u_0+\nabla^\perp G'R_1\theta_0) |_{L^\infty}\leq Ct^2 (|F_0|_{B^{\frac{1}{2}}_{4,1}}+|F_0|_{B^{\frac{1}{2}}_{4,1}}^2).$$
Now we will choose $F_0$ such that 
$$|F_0|_{L^\infty}<\epsilon$$
$$|G'R_1(F_0)|_{L^\infty}=c\epsilon \log N$$
$$|F_0|_{B^{\frac{1}{2}}_{4,1}}=C\epsilon \log N$$ with $N$ a constant to be chosen and $C$ is a universal constant.  
We then see:
$$|F|_{L^\infty}\geq ct\epsilon\log N - Ct^2\epsilon^2(\log N)^2-C\epsilon$$ and we then choose $t=\epsilon$ and $N$ suitably to give: 
$$|F(t=\epsilon)|_{L^\infty}>c$$ for some constant $c>0$ independent of $\epsilon$. 

\vspace{3mm}

\emph{Step 3: Choosing the right initial data.} The proof will be finished once we exhibit an $\theta_0\in W^{1,\infty}(\mathbb{T}^2)$ with $F_0=\nabla^\perp\theta_0$ satisfying the properties above with the properties above. First, assume without a loss of generality that $G'(0)=c>0$. A simple example of such an $\theta_0$ is the function defined on the periodic box $[-1,1]^2$ by $$\theta_0(x_1,x_2)=\frac{\epsilon}{100}\,\, (\phi_N\ast  |\cdot|) \chi(x_1)$$ with $\chi(x)=\chi(-x)$, $\chi\equiv 1$ on $[-\frac{1}{4},\frac{1}{4}]$, $\chi\equiv 0$ outside of $[-\frac{1}{2},\frac{1}{2}],$ and $\chi\in C^\infty$ and $\phi_N(x)= N\exp(-x^2N^2).$ The reason for such a choice is that when $F_0$ is a function of only $x_1,$ $R_1$ becomes the Hilbert transform in $x_1$ and it is well known that the Hilbert transform of $\partial_{x_1} |x_1|=\text{sgn}(x_1)$ is a constant multiple of $\log x_1$ near $x_1=0$. Hence, without convolving with $\phi_{N}$, we see that $\nabla \theta_0$ is bounded uniformly by $\epsilon$ and $G'R_1\nabla\theta_0$ is like $\log x_1$ in a neighborhood of the origin. We leave the rest to the interested reader.  

\end{proof}

\subsection{The Boussinesq System}

Recall the two-dimensional Boussinesq system \cite{MB}:
\begin{equation} \label{Boussinesq1} \partial_t u + (u \cdot \nabla) u + \nabla p = \left( \begin{array}{cc}
0 \\
\rho \\
\end{array} \right), \end{equation}
\begin{equation}\label{Boussinesq2} \partial_t \rho+ u\cdot\nabla\rho=0. \end{equation}

The global well-posedness of this system is an outstanding open problem in the study of incompressible fluid equations. The system models the effects of temperature variations in a fluid and the buouancy effects they induce. It is known that control on $\nabla u$ or $\nabla\rho$ in $L^\infty$ is enough to rule out singularity formation.  
One may ask the following question:

\vspace{3mm}
{\bf Question:} Is it possible to prove a (local) a-priori bound on the vorticity in $L^\infty$ as in the case $\rho\equiv 0$?
\vspace{3mm}

The same question may be asked about the $L^\infty$ norm of $\nabla\rho$ and $\nabla u$. To clearly see the effect of adding density to the problem, we will look at perturbations of the stationary solution $(\rho^*,u^*)=(-y,0).$ Note that this stationary solution is actually linearly stable (see \cite{DR}). Using Theorem \ref{LowerBoundTheorem}, we will prove that vorticity which is initially very close to $(\rho^*,u^*)$ immediately moves far away in an $L^\infty$ sense. 

\begin{thm}
There exists a fixed constant $c>0$ and a sequence of initial data $u_0^\epsilon,\rho_0^\epsilon\in C^\infty(\mathbb{T}^2)$ with $|\theta_0^\epsilon-y|_{W^{1,\infty}}<\epsilon$ and $|\omega_0^\epsilon|_{L^\infty}<\epsilon$ but such that the associated unique local solution $\rho^\epsilon$ to \eqref{Boussinesq1}-\eqref{Boussinesq2} satisfies:
$$|\omega^\epsilon(t)|_{L^\infty}>c$$ for some $t<\epsilon$. 
\end{thm}

\begin{proof}
The proof will proceed in a similar fashion to the previous two proofs. We will begin by analyzing the linearized equation and then see how to control the non-linear terms.

\emph{Step 1: Linear Analysis}
 
\noindent First we write the equation for perturbations of the stationary solution $(-y,0)$ in the vorticity formulation:
\begin{equation}\label{B1}
\partial_t \omega+ u\cdot\nabla\omega=\partial_{x_1} \rho,
\end{equation}
\begin{equation}\label{B2}
\partial_t\rho+u\cdot\nabla\rho =-u_2.
\end{equation}
Next we consider the linearized system:
$$\partial_t\omega=\partial_{x_1}\rho,$$
$$\partial_t\rho= -u_2,$$
$$u_2=\partial_{x_1}\Delta^{-1}\omega,$$
Upon differentiating the $\rho$ with respect to $x$, we see:
$$\partial_{tt}\omega = R_1^2\omega.$$
Going to the Fourier side we see $$\partial_{tt}\hat\omega = -\frac{\xi_1^2}{|\xi|^2}\hat\omega.$$
Hence, $$\omega= \exp(R_1t)A + \exp(-R_1t)B$$ with $A+B=\partial_{x_1}\omega_0$ and $R_1(A-B)=\partial_{x_1}\rho_0$ which then implies that $A-B=(-\Delta)^{-\frac{1}{2}}\rho_0.$ It is now clear that the linearized equation is ill-posed on $L^\infty$.

\emph{Step 2: Lower bounds for the non-linear problem}
\noindent In line with the linear analysis, we begin by differentiating \eqref{B2} with respect to $x$. 
Then we see: 
$$\partial_t\omega+u\cdot\nabla\omega= \partial_{x_1}\rho$$
$$\partial_t\partial_{x_1}\rho+u\cdot\nabla\partial_{x_1}\rho+\partial_{x_1} u\cdot\nabla\rho = R_1^2\omega$$
Now we consider the equation along the Lagrangian flow-map associated to $u$, which we call $\Phi$, and we see:
$$\partial_t(\omega\circ\Phi)=\partial_{x_1}\rho\circ\Phi$$
$$\partial_t (\partial_{x_1}\rho\circ\Phi) +(\partial_{x_1} u\cdot\nabla\rho)\circ\Phi=R_1^2(\omega\circ\Phi) + [R_1^2,\Phi]\omega.$$ Let us call the linear group associated to the linear system in Step 1 $\exp(Lt).$ Then we see:

\begin{align}
    \partial_t\Big\{\exp(-Lt)\begin{bmatrix}
           \omega\circ\Phi \\
              \partial_{x_1}\rho\circ\Phi
         \end{bmatrix}\Big\}=\exp(-Lt)\begin{bmatrix} 0 \\
[R,\Phi]\omega-(\partial_{x_1} u\cdot\nabla\rho)\circ\Phi
\end{bmatrix}.
  \end{align}
This then implies:

$$\begin{bmatrix} \omega\circ\Phi \\
\partial_{x_1}\rho\circ\Phi\end{bmatrix}=\exp(Lt) \begin{bmatrix} \omega_0\\ \partial_{x_1}\rho_0\end{bmatrix}+\int_0^t\exp(L(t-s))\begin{bmatrix} 0 \\
[R,\Phi]\omega-(\partial_{x_1} u\cdot\nabla\rho)\circ\Phi
\end{bmatrix}. $$
Now we suppose there exists a sequence $(\omega_0^\epsilon,\rho_0^\epsilon)\in C^\infty_c$ such that:
$$\| \begin{bmatrix} \omega_0^\epsilon\\ \partial_{x_1}\rho_0^\epsilon\end{bmatrix}\|_{L^\infty}<\epsilon,$$
$$\|R_1 \begin{bmatrix} \omega_0^\epsilon\\ \partial_{x_1}\rho_0^\epsilon\end{bmatrix}\|_{L^\infty}\geq c\epsilon \log N,$$
$$\|\begin{bmatrix} \omega_0^\epsilon\\ \partial_{x_1}\rho_0^\epsilon\end{bmatrix} \|_{B^{\frac{1}{2}}_{4,1}}\leq C\epsilon\log N.$$
In fact, we have already shown the existence of such functions in the proof of Theorem \ref{SQGThm}.
Then we get (after a simple computation):
$$|\nabla\rho^\epsilon|_{L^\infty} \geq c\epsilon t \log N-C\epsilon^2 t^2(\log N)^2 - t\epsilon\log N|\nabla\rho^\epsilon|_{L^\infty_{x,t}}.$$
To deal with the last term we simply assume without a loss of generality that $|\nabla\rho^\epsilon|_{L^\infty_{x}}<\frac{c}{2}$ uniformly on $t\in[0,\epsilon]$ for otherwise we  are done. 
Then we see (as usual): $$|\nabla\rho^\epsilon|_{L^\infty} \geq \frac{c}{2}\epsilon t \log N-C\epsilon^2 t^2(\log N)^2$$
Now if we take $t\epsilon\log N$ smaller than the very small constant $\frac{c}{100C}$ we are done. 

\end{proof}

\section{Conclusion}

In section 3 we prove a linear ill-posedness result for general transport equations with Lipschitz velocity fields and singular integral forcing. As a consequence we proved strong ill-posedness for a particular linear equation. We saw in the previous sections that proving an  $L^\infty$ mild ill-posedness result for non-linear equations is possible when three conditions are satisfied: first, that the velocity field be related to the advected quantity by a degree -1 operator (which is the case for the vorticity equation for example). Second, that the equation be locally well-posed in the critical Besov space which imbeds in $L^\infty$. Finally, that the non-local operator on the right-hand side satisfy Assumption \ref{Assumption1}. We then use the method to prove strong ill-posedness of the Euler equations in $C^k$ for integer $k.$ Finally, as examples of how this method can be used for other equations, we prove mild ill-posedness for the Oldroyd-B viscoelastic system, the surface quasi-geostrophic equation, and the Boussinesq system.

\section{Acknowledgements}

Both authors  were  partially supported by NSF grant DMS-1211806. The first author was also supported by NSF grant DMS-1402357. The first author thanks Gautam Iyer for motivating him to "write down" the $C^1$ results. The authors would also like thank the anonymous referees for their comments which greatly improved the paper.

\end{document}